\documentclass[11pt,a4paper,intlimits,reqno]{amsart}

\usepackage{amsmath}
\usepackage{amsopn,amsfonts,amssymb,amsthm}

\usepackage{mathrsfs}
\usepackage[T1]{fontenc}
\usepackage[english]{babel}
\usepackage[utf8]{inputenc}
\usepackage{epsfig}
\input{epsfx}
\usepackage{graphicx}

\newtheorem{thm}{Theorem}[section]
\newtheorem{prop}[thm]{Proposition}
\newtheorem{lem}[thm]{Lemma}

\newtheorem{cor}[thm]{Corollary}

\numberwithin{equation}{section}

\newcommand{\N}{\mathbb{N}}
\newcommand{\Z}{\mathbb{Z}}
\newcommand{\C}{\mathbb{C}}

\newcommand{\R}{\mathbb{R}}

\newcommand{\D}{\mathbb{D}}

\DeclareMathOperator{\im}{Im}
\DeclareMathOperator{\re}{Re}
\DeclareMathOperator{\dist}{dist}
\DeclareMathOperator{\diam}{diam}
\DeclareMathOperator{\HD}{HD}

\DeclareMathOperator{\erior}{Int}

\newcommand{\sms}{\setminus}
\newcommand{\leeq}{\leqslant}
\newcommand{\greq}{\geqslant}

\newcommand{\ve}{\varepsilon}
\newcommand{\la}{\lambda}
\newcommand{\mc}{\mathcal}
\newcommand{\vp}{\varphi}

\begin{document}

\title[On the directional derivative of the Hausdorff dimension]{On the directional derivative of the Hausdorff dimension of quadratic polynomial Julia sets at 1/4}

\author{Ludwik Jaksztas}

\address{Faculty of Mathematics and Information Sciences, Warsaw University of Technology, ul. Koszykowa 75, 00-662 Warsaw, Poland, ORCID: 0000-0002-9283-8841}
\email{jaksztas@impan.pl}
\subjclass[2010]{Primary 37F45, Secondary 37F35}
\keywords{Hausdorff dimension, Julia set, quadratic family}

\begin{abstract}
Let $d(\ve)$ and $\mc D(\delta)$ denote the Hausdorff dimension of the Julia sets of the polynomials $p_\ve(z)=z^2+1/4+\ve$ and $f_\delta(z)=(1+\delta)z+z^2$ respectively.

In this paper we will study the directional derivative of the functions $d(\ve)$ and $\mc D(\delta)$ along directions landing at the parameter $0$, which corresponds to $1/4$ in the case of family $z^2+c$. We will consider all directions, except the one $\ve\in\R^+$ (or two imaginary directions in the $\delta$ parametrization) which is outside the Mandelbrot set and is related to the parabolic implosion phenomenon.

We prove that for directions in the closed left half-plane the derivative of $d$ is negative. Computer calculations show that it is negative except a cone (with opening angle approximately $150^\circ$) around $\R^+$.

\end{abstract}

\maketitle


\section{Introduction}\label{sec:introd}

Let $f$ be a polynomial in one complex variable of degree at least $2$. The \emph{filled-in Julia set} $K(f)$ we define as the set of all points that do not escape to infinity under iteration of $f$, i.e.
   $$K(f)=\{z\in\C:f^n(z)\nrightarrow\infty\}.$$
It is a compact set whose boundary is called the \emph{Julia set}. So, let us write
$$J(f):=\partial K(f).$$

In this paper we will consider two families of quadratic polynomials. For technical reasons, we slightly modify the classical families $z^2+c$ and $\lambda z +z^2$. We will deal with
\begin{equation*}
 \begin{array}{ll}
   p_\ve(z)=z^2+1/4+\ve,\\
   f_\delta(z)=(1+\delta)z+z^2,
 \end{array}
\end{equation*}
where $\ve, \delta\in\C$, however we will consider parameters close to $0$. Obviously $\ve=0$ and $\delta=0$ correspond to $c=1/4$ and $\lambda=1$ respectively, and all results can be easily transferred to the parametrizations $z^2+c$ and $\lambda z +z^2$.

If $\tau_\delta(z)=z+(1+\delta)/2$, then we have
\begin{equation}\label{eq:conj}
\tau_\delta\circ f_\delta\circ\tau_\delta^{-1}(z)=z^2+\frac14-\frac{\delta^2}{4}=p_{-\delta^2/4}(z).
\end{equation}
Thus the polynomials $f_\delta$ and $p_\ve$ are conjugated by a similarity if and only if
\begin{equation}\label{eq:ed}
   \ve=-\delta^2/4.
\end{equation}
In particular we see that $f_0$ is conjugated to $p_{0}$, and $f_\delta$ is conjugated to $f_{-\delta}$.

\begin{figure}[!h]
\begin{center}
\includegraphics[width=6.2cm]{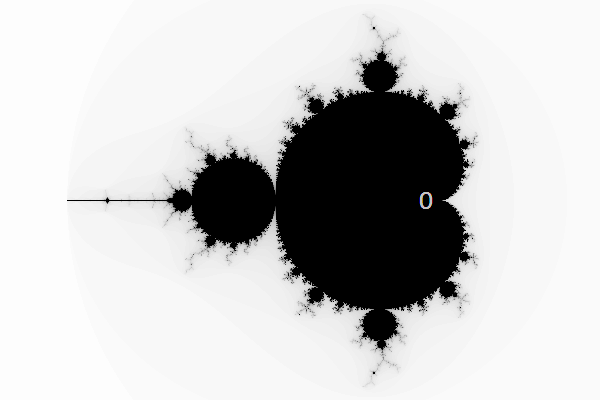} \includegraphics[width=6.2cm]{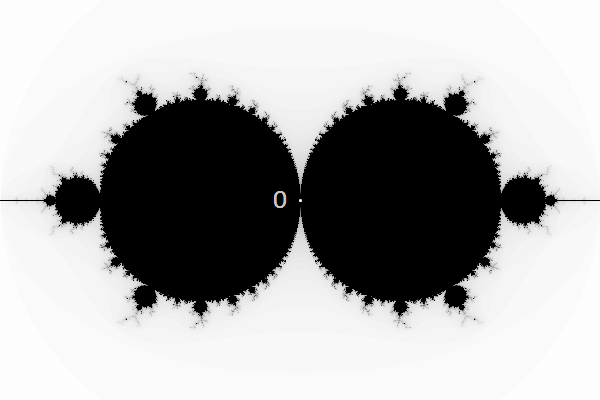}
\caption{The Mandelbrot sets $M$ and $\mc M$ }
\end{center}
\end{figure}

We define the \emph{Mandelbrot sets} as follows:
\begin{equation*}
 \begin{array}{ll}
   M:=\{\ve\in\C:p_\ve^n(0)\nrightarrow\infty\},\\
   \mc M:=\{\delta\in\C:f_\delta^n(-1/2-\delta/2)\nrightarrow\infty\}.
 \end{array}
\end{equation*}
Note that $0$ and $-1/2-\delta/2$ are the only critical points of $p_\ve$ and $f_\delta$ respectively.
Equivalently $M$ and $\mc M$ can be defined as the sets of all parameters for which respective Julia sets are connected.

We will use the following abbreviations:
\begin{equation*}
 \begin{array}{ll}
   J_\ve:=J(p_\ve),\;\; K_\ve:=K(p_\ve),\\
   \mc J_\delta:=J(f_\delta),\;\; \mc K_\delta:=K(f_\delta).
 \end{array}
\end{equation*}

Let $d(\ve):=\HD(J_\ve)$ and $\mc D(\delta):=\HD(\mc J_\delta)$ denote the Hausdorff dimension of the Julia sets.
We will study properties of the functions
$$\ve\mapsto d(\ve)\;\;\; \textrm{ and }\;\;\;\delta\mapsto \mc D(\delta).$$

Recall that a polynomial $f:\overline\C\rightarrow\overline\C$ is called
\emph{hyperbolic (expanding)} if there exists $n\greq1$ such that $|(f^n)'(z)|>1$ for every $z\in J(f)$.

The function $\mc D(\delta)$ is real-analytic on each hyperbolic component of $\erior\mc M$ (consisting of parameters
related to hyperbolic maps) as well as on the exterior of $\mc M$ (see \cite{R}). In particular $\mc D(\delta)$ is
real-analytic on $\mc M_0^+=B(1,1)$ and $\mc M_0^-=B(-1,1)$, the components of $\erior\mc M$ that consist of parameters $\delta$ for which the polynomial $f_\delta$ has an attracting fixed point.

Note that $0$ and $-\delta$ are the fixed points of $f_\delta$. Thus we have $f'_\delta(0)=1+\delta$ and $|1+\delta|<1$ for $\delta\in B(-1,1)$, whereas $f'_\delta(-\delta)=1-\delta$ and $|1-\delta|<1$ for $\delta\in B(1,1)$.

Analogously the function $d(\ve)$ is real-analytic on each hyperbolic component of $M$, in particular on the largest component $M_0$ bounded by the so called main cardioid. The component $M_0$ is related to the components $\mc M_0^+$, $\mc M_0^-$,
i.e. $\ve\in M_0$ if and only if $\ve=-\delta^2/4$ where $\delta\in\mc M_0^+$ (or $\delta\in\mc M_0^-$).

We have $0\in\partial M$ and $0\in\partial\mc M_0^+\cap\partial\mc M_0^-$ (thus $0\in M$ and $0\in\mc M$). Moreover the polynomials $p_0(z)=z^2+1/4$ and $f_0(z)=z+z^2$ have parabolic fixed points with one petal, i.e. $p'_0(1/2)=1$ and $f'_0(0)=1$.

Let us assume that $\alpha\in[-\pi/2,3\pi/2)$. Write
$$\mc R(\alpha):=\{z\in\C^*:\alpha=\arg z\}.$$

We will study properties of the functions $d(\ve)$, $\mc D(\delta)$ when the parameters $\ve\in M$, $\delta\in \mc M$ tend to $0$ along the rays $\mc R(\alpha)$. So, we will consider all directions except $\ve\in\mc R(0)$, and $\delta\in\mc R(\pm\pi/2)$ (cf. Figure 1.). Note that these exceptional directions are related to the parabolic implosion phenomenon, and the Hausdorff dimension of Julia sets is not continuous at $0$ along them (see \cite{DSZ}).

Let us first consider the case $\ve\in\mc R(\pi)$ (i.e. $\ve\in\R^-$). O. Bodart and M. Zinsmeister proved in \cite{BZ} the following theorem:\vspace{2.5mm}
\newline
\textbf{Theorem BZ.} \emph{The function $d$ restricted to the real axis is left-sided continuous at $0$.}
\vspace{2mm}

Analogously, we see that $\mc D|_\R$ is continuous at $0$ from both sides (cf. \ref{eq:ed}).

In \cite{HZi}, G. Havard and M. Zinsmeister studied more precisely behavior of $d$ on the left side of $0$. They
proved that: \vspace{2.5mm}
\newline
\textbf{Theorem HZ.} \emph{There exist $c_0<0$ and $K>1$ such that for every $\ve\in(c_0, 0)$}
   $$\frac{1}{K}(-\ve)^{d(0)-\frac{3}{2}}\leeq d'(\ve)\leeq K(-\ve)^{d(0)-\frac{3}{2}}.$$

We have $1<d(0)<1.295<1.5$ (see \cite{Zd} and \cite{HZ}). Therefore $d'(\ve)\rightarrow+\infty$ when $\ve\rightarrow0^-$.
Related statement for the function $\mc D(\delta)$ follows from Theorem \ref{thm:direction2} (see below), but let us note that in this case we obtain $\mc D'(\delta)\rightarrow0$.

Next, it was also proven in \cite{Jii} that:\vspace{2.5mm}
\newline
\textbf{Theorem J.} \emph{There exists $c_1<0$ such that}
   $$d''(\ve)>0,$$
\emph{where $c\in(c_1,0)$ (i.e. $d$ is a convex function on the interval $(c_1,0)$). Moreover $d''(\ve)\rightarrow\infty$
when $\ve\rightarrow{0}^-$.}
\vspace{2mm}

Let us now consider non-real directions. First we adapt to our situation the definition from \cite{Mii}. We say that $f_{\delta_n}\rightarrow f_0$ horocyclically if $\delta_n\rightarrow0$ and
\begin{equation}\label{eq:horo}
   {\im^2 \delta_n}/{\re \delta_n}\rightarrow0,
\end{equation}
when $n\rightarrow\infty$. Moreover we say that $p_{\ve_n}\rightarrow p_0$ horocyclically if $\ve_n\rightarrow0$ and there exist $\delta_n$ such that $\ve_n=-\delta^2_n/4$ and (\ref{eq:horo}) holds.

The following fact follows from theorem proved by C. McMullen in \cite[Theorem 11.2]{Mii} (see also \cite{BT}):\vspace{2.5mm}
\newline
\textbf{Theorem M.} (1)\emph{ If $p_{\ve_n}\rightarrow p_0$ horocyclically, then}
   $$\lim_{n\rightarrow\infty}d(\ve_n)=d(0).$$
\emph{In particular the limit exists if $\ve_n\rightarrow0$ along any ray $\mc R(\alpha)$, where $\alpha\neq0$.}

(2)\emph{ If $f_{\delta_n}\rightarrow f_0$ horocyclically, then}
   $$\lim_{n\rightarrow\infty}\mc D(\delta_n)=\mc D(0).$$
\emph{In particular the limit exists if $\delta_n\rightarrow0$ along any ray $\mc R(\alpha)$, where $\alpha\neq\pm\pi/2$.}\vspace{2mm}

In order to state our main results we need two definitions.

Let $F$ be a real function defined on a domain $U\in\C$. If $z\in U$ and $v\in\C^*$, then
   $$F'_v(z):=\lim_{h\rightarrow0}\frac{F(z+hv)-F(z)}{h},$$
if the limit exists.

For $\vartheta\in\R$ we write
\begin{equation}\label{eq:I}
\Omega(\vartheta):=\sqrt{\vartheta^2+1}\int_{0}^{\infty} \Big(2-\frac{x\sinh x+\vartheta x\sin\vartheta x}{\cosh x-\cos\vartheta x}\Big)
\Big(\frac{1}{\cosh x-\cos\vartheta x}\Big)^{d(0)}dx.
\end{equation}
Note that this integral converges for every $\vartheta\in\R$, because $d(0)<3/2$.

\begin{figure}[!h]
   \centering
        \includegraphics[scale=0.72]{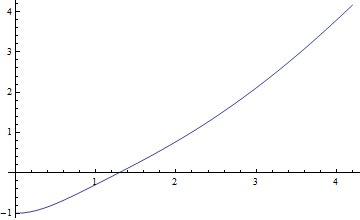}
        \caption{Graph of the function $\Omega(\vartheta)$ under assumption that $d(0)=1.08$.}
\end{figure}

The main Theorem in this paper is:
\begin{thm}\label{thm:direction}
There exists $A>0$ such that for every $\beta\in(0,2\pi)$ we have
   $$\lim_{t\rightarrow0^+}\frac{d'_u(tu)}{t^{d(0)-3/2}}=A \cdot\Omega(\vartheta),$$
where $u=e^{i\beta}$ and $\vartheta=\cot(\beta/2)$. Moreover $\Omega(\vartheta)<0$ for $\vartheta\in[-1,1]$ (i.e. $\beta\in[\pi/2,3\pi/2]$).
\end{thm}

Thus we see that $\lim_{t\rightarrow0^+}d'_v(tv)=\pm\infty$, provided $\Omega(\vartheta)\neq0$.

Obviously $\Omega(\vartheta)$ is an even function, so we conclude that there exists a maximal interval $(-\vartheta_0,\vartheta_0)\supset[-1,1]$ on which $\Omega(\vartheta)<0$ (analogously we have $(\pi-\beta_0,\pi+\beta_0)\supset[\pi/2,3\pi/2]$).
Numerically made picture of the graph of $\Omega(\vartheta)$ (see Figure 2.) suggests that it is an increasing function for $\vartheta>0$, whereas $\vartheta_0$ is finite and close to $1.3$ (so $\beta_0$ is close to $7\pi/12$), but we do not know how to prove that.

\

Let $\beta\in(0,2\pi)$ and let $\alpha\in(-\pi/2,\pi/2)\cup(\pi/2,3\pi/2)$. We will assume that $\beta=2\alpha+\pi$ or $\beta=2\alpha-\pi$. Thus, if $v=e^{i\alpha}$, $u=e^{i\beta}$ and $\delta=tv$, $t>0$, then $-\delta^2=-t^2v^2=t^2u$. So (\ref{eq:ed}) leads to $\mc D(tv)=d(t^2u/4)$, hence
   $$\mc D_v'\big(tv\big)=\frac t2\, d_u'\Big(\frac{t^2}{4}u\Big)\;\;\textrm{ and }\;\;
   \frac{\mc D_v'(tv)}{(t/2)^{2\mc D(0)-2}}=\frac{ d_u'(t^2u/4)}{(t^2/4)^{^{d(0)-3/2}}},$$
where $tv\in\mc M_0^-\cup\mc M_0^+$. Note that $\tan(\alpha)=-\cot(\beta/2)$ and $\Omega(\vartheta)=\Omega(-\vartheta)$. Thus, the following Theorem is equivalent to Theorem \ref{thm:direction}.

\begin{thm}\label{thm:direction2}
There exists $\mc A>0$ such that for every $\alpha\in(-\pi/2,\pi/2)\cup(\pi/2,3\pi/2)$ we have
   $$\lim_{t\rightarrow0^+}\frac{\mc D'_v(tv)}{t^{2\mc D(0)-2}}=\mc A \cdot \Omega(\vartheta),$$
where $v=e^{i\alpha}$ and $\vartheta=\tan(\alpha)$. Moreover $\Omega(\vartheta)<0$ for $\vartheta\in[-1,1]$ (i.e. $\alpha\in[-\pi/4,\pi/4]\cup[3\pi/4,5\pi/4]$).
\end{thm}
Since $\mc D(0)<3/2$, we conclude that $\lim_{t\rightarrow0^+}\mc D'_v(tv)=0$ for every $|v|=1$, except $v=\pm i$.
Note that $A=2^{2\mc D(0)-2}\mc A$. 

The rest of the paper is devoted to proving Theorem \ref{thm:direction2}, and then Theorem \ref{thm:direction} will follow. So, we will deal with the family $f_\delta$, however we will use the fact that $f_\delta$ is conjugated to $p_{-\delta^2/4}$.

The proof is based upon ideas introduced in \cite{HZ}, and developed in \cite{J}, \cite{Ji}, \cite{JZ}. In the latter three papers, we estimated the derivative of the Hausdorff dimension of $J(z^2+c)$, when the real parameter $c$ tends to a parabolic parameter with two petals. However, some properties of the holomorphic motion are easier to study in the one petal case (we do not need the symmetry of the Julia set). So, this gives us a possibility to estimate the directional derivative along non-real directions.

Since the polynomials $f_\delta$ and $f_{-\delta}$ are conjugated, we will only consider the case $\delta\in\mc M_0^+$ (i.e. $\alpha\in(-\pi/2,\pi/2)$).

Notation.
$X\asymp Y$ means that $K^{-1}\leeq X/Y\leeq K$, where constant $K>1$ does not depend on $X$ and $Y$ under consideration.

\section{Thermodynamic formalism}\label{sec:formalism}

The main goal of this section is to establish formula for directional derivative of the Hausdorff dimension (see Proposition \ref{prop:d}).

We shall repeat after \cite[Section 2]{Jii} the basic notions.

If $\delta\in\mathcal M$ then there exists the function $\vp_\delta:\C\sms\overline\D\rightarrow\C\sms \mc K_\delta$ (called the
B\"{o}ttcher coordinate) which is holomorphic, bijective, asymptotic to identity map at infinity, and conjugating
$T(s)=s^2$ to $f_\delta$ (i.e. $\vp_\delta\circ T=f_\delta\circ\vp_\delta$). Since the Julia set $\mc J_\delta$ is
a Jordan curve for $\delta\in\mathcal M^+_0\cup\{0\}$, the function $\vp_\delta:\C\sms\overline\D\rightarrow\C\sms\mc K_\delta$ has homeomorphic extension to
$\partial\D$ (Carath\'{e}odory's Theorem) and $\vp_\delta$ conjugates $T|_{\partial\D}$ to $f_\delta|_{\mc J_\delta}$.

The map $(\delta,s)\mapsto\vp_\delta(s)$ gives a holomorphic motion for $\delta\in\mathcal M_0^+$ (see \cite{Hub}). Thus, the
functions $\vp_\delta$ are quasiconformal, and then also H\"{o}lder continuous, whereas $\delta\mapsto\vp_\delta(s)$ are
holomorphic for every $s\in\C\sms\D$ (in particular for $s\in\partial\D$).

Now we use the thermodynamic formalism, which holds for hyperbolic rational maps.
Let $X=\partial\D$, $T(s)=s^2$, and let $\phi:X\rightarrow\R$ be a potential function of the form
$\phi=-\tau\log|f_\delta'(\vp_\delta)|$, for $\delta\in \mc M_0^+$ and $\tau\in\R$.

{\em The topological pressure} can be defined as follows:
   $$P(T,\phi):=\lim_{n\rightarrow\infty}\frac{1}{n}\log\sum_{\overline x\in{T^{-n}(x)}}e^{S_n(\phi(\overline x))},$$
where $S_n(\phi)=\sum_{k=0}^{n-1}\phi\circ T^k$. The limit exists and does not depend on $x\in\partial\D$.
If $\phi=-\tau\log|f_\delta'(\vp_\delta)|$ and $\vp_\delta(\overline x)=\overline z$, then $e^{S_n(\phi(\overline
x))}=|(f_\delta^n)'(\overline z)|^{-\tau}$, hence
   $$P(T,-\tau\log|f_\delta'(\vp_\delta)|)=\lim_{n\rightarrow\infty}\frac{1}{n}\log\sum_{\overline z\in{f_\delta^{-n}(z)}}|(f_\delta^n)'(\overline z)|^{-\tau}.$$
The function $\tau\mapsto P(T,-\tau\log|f_\delta'(\vp_\delta)|)$ is strictly decreasing from $+\infty$ to $-\infty$. So, there exists a
unique $\tau_0$ such that $P(T,-\tau_0\log|f_\delta'(\vp_\delta)|)=0$. By Bowen's Theorem (see \cite[Corollary 9.1.7]{PU} or
\cite[Theorem 5.12]{Z}) we obtain
  \begin{equation*}\label{eq:acm}
    \tau_0=\mc D(\delta).
  \end{equation*}
Thus, we have $P(T,-\mc D(\delta)\log|f_\delta'(\vp_\delta)|)=0$. Write $\phi_\delta:=-\mc D(\delta)\log|f_\delta'(\vp_\delta)|$.

{\em The Ruelle operator} or {\em the transfer operator} $\mathcal{L}_{\phi}:C^0(X)\rightarrow C^0(X)$, is defined as
  \begin{equation*}\label{eq:L}
    \mathcal{L}_{\phi}(u)(x):=\sum_{\overline x\in{T^{-1}(x)}}u(\overline x)e^{\phi(\overline x)}.
  \end{equation*}
The Perron-Frobenius-Ruelle theorem \cite[Theorem 4.1]{Z} asserts that $\beta=e^{P(T,\phi)}$ is a single
eigenvalue of $\mathcal{L}_{\phi}$ associated to an eigenfunction $\tilde h_{\phi}>0$. Moreover there exists a
unique probability measure $\tilde \omega_{\phi}$ such that $\mathcal{L}_{\phi}^*(\tilde
\omega_{\phi})=\beta\tilde\omega_{\phi}$, where $\mathcal{L}_{\phi}^*$ is dual to $\mathcal{L}_{\phi}$.

For $\phi=\phi_\delta$ we have $\beta=1$, and then $\tilde \mu_{\phi_\delta}:=\tilde
h_{\phi_\delta}\tilde\omega_{\phi_\delta}$ is a $T$-invariant measure called an equilibrium state (we assume that this
measure is normalized).
We denote by $\tilde{\omega_\delta}$ and $\tilde{\mu_\delta}$ the measures $\tilde\omega_{\phi_\delta}$ and
$\tilde\mu_{\phi_\delta}$ respectively (measures supported on the unit circle). Next, we take
$\mu_{\delta}:=(\vp_{\delta})_*\tilde{\mu_{\delta}}$, and $\omega_{\delta}:=(\vp_{\delta})_*\tilde{\omega_{\delta}}$ (measures supported on the Julia set $\mc J_\delta$).

So, the measure $\mu_\delta$ is $f_\delta$-invariant, whereas $\omega_\delta$ is called {\em $f_\delta$-conformal measure with exponent
$\mc D(\delta)$}, i.e. $\omega_\delta$ is a Borel probability measure such that for every Borel subset $A\subset\mc J_\delta$,
  \begin{equation*}\label{eq:SP}
    \omega_\delta(f_\delta(A))=\int_A|f_\delta'|^{\mc D(\delta)}d\omega_\delta,
  \end{equation*}
provided $f_\delta$ is injective on $A$.

It follows from \cite[Proposition 6.11]{Z} or \cite[Theorem 4.6.5]{PU} that for every H\"{o}lder $\psi$ and $\hat\psi$ at every $t\in\R$, we have
  \begin{equation*}\label{eq:SP1}
    \frac{\partial}{\partial t}P(T,\psi+t\hat\psi)=\int_X \hat\psi \:d\tilde{\mu}_{\psi+t\hat\psi}.
  \end{equation*}

Let us consider parameters of the form $\delta=te^{i\alpha}=tv$, where $t>0$. Since $\tau=\mc D(tv)$ is the unique zero of the pressure function, for the potential $\phi=-\tau\log|f'_{tv}(\varphi_{tv})|$, the implicit function theorem combined with the above formula leads to (see \cite[Proposition 2.1]{HZi} or \cite[Proposition 2.1]{J}):

\begin{prop}\label{prop:d}
If $\alpha\in(-\pi/2,\pi/2)$ and $v=e^{i\alpha}$, then for every $t>0$ such that $tv\in\mc M_0^+$ we have
\begin{equation}\label{eq:d}
\mc D_v'(tv)=
\frac{-\mc D(tv)}{\int_{\partial\D}\log|f_{tv}'(\vp_{tv})|d\tilde\mu_{tv}}\int_{\partial\D}\frac{\partial}{\partial
t}\log|f_{tv}'(\vp_{tv})|d\tilde\mu_{tv}.
\end{equation}
\end{prop}

\section{Fatou coordinates}\label{sec:fatou}

Now we are going to introduce coordinates in which the family $f_\delta$ is close to the translation by $1$, on the Julia set, near to the fixed points $0$ and $-\delta$.
We will call these coordinates the Fatou-coordinates (repelling and attracting), even if they do not conjugate to an exact translation.
We shall use results of Xavier Buff and Tan Lei (see \cite{BT}).

We have
   $$f_\delta(z)=z+z(z+\delta).$$
The time-one map of the flow
   $$\dot z=f_\delta(z)-z=z(z+\delta),$$
gives a good approximate of $f_\delta$ (close to the fixed points).

The function
\begin{equation}\label{eq:Psi}
   \Psi_\delta(w):=\frac{\delta}{e^{-w\delta}-1}
\end{equation}
is a solution of the above equation, whereas the formal inverse is given by
\begin{equation*}
   \Phi_\delta(z):=\frac{1}{\delta}\log\Big(1-\frac{\delta}{z+\delta}\Big)= -\frac{1}{\delta}\log\Big(1+\frac{\delta}{z}\Big).
\end{equation*}
Let us also write
   $$\Psi_0(w)=-\frac1w\;\;\textrm{ and }\;\; \Phi_0(z)=-\frac1z.$$

If $z\in \C^*$, then we shall assume that $\arg z\in[-\frac12\pi,\frac32\pi)$.
Let us define:
\begin{equation*}
 \begin{array}{ll}
   S^+(\theta,r):=\{z\in\C^*:|\arg z|<\theta,\:|z|<r\},\\
   S^-(\theta,r):=\{z\in\C^*:|\arg z-\pi|<\theta,\:|z|<r\}.
 \end{array}
\end{equation*}
Set $S^\pm(\theta):=S^\pm(\theta,\infty)$, $\hat S_\la^\pm(\theta):=\hat S_\la^\pm(\theta,\infty)$, and (cf. \cite{BT})
   $$S^\pm(\theta)_R:=S^\pm(\theta)\cap\{z\in\C:|\re z|>R\}.$$

It follows from \cite[Proposition 2.6]{BT} that:
\begin{lem}\label{lem:Psiunif}
  If $\alpha\in(-\pi/2,\pi/2)$ then for every $\theta\in(0,\pi/2-|\alpha|)$ we have
     $$\Psi_\delta\rightarrow \Psi_0$$
  uniformly (not just locally uniformly) on $S^\pm(\theta)$, where $\delta\rightarrow0$ and $\alpha=\arg\delta$.
\end{lem}

Let $f^{-1}_{\delta}$ be the inverse branch of $f_\delta$ that keeps fixed points $0$, $-\delta$ (we can assume that $f^{-1}_{\delta}$ is defined on the set $\{z\in\C:\re(z)>\re(f_\delta(c_\delta))\}$, where $c_\delta$ is the only critical point of $f_\delta$). Now we prove the main result of this section:

\begin{lem}\label{lem:trans}
  If $\alpha\in(-\pi/2,\pi/2)$ then for every $\theta\in(0,\pi/2-|\alpha|)$ and $\varepsilon>0$ there exist $R>0$ and $\eta>0$ such that $\Psi_\delta$ has a local inverse map $\Phi_\delta$ and
     $$F^{-1}_\delta=\Phi_\delta\circ f^{-1}_\delta\circ \Psi_\delta$$
  is well defined and univalent on $S^-(\theta)_R$, mapping $S^-(\theta)_R$ into $S^-(\theta)_R$. Moreover
    \begin{equation*}
       \sup_{n\in\N,\:w\in S^-(\theta)_R} |F_\delta^{-n}(w)-(w-n)|<\varepsilon n,
    \end{equation*}
   where $0<|\delta|<\eta$ and $\alpha=\arg\delta$ or $\delta=0$.

   There ia a similar statement for $S^+(\theta)_R$ replacing triple $(F^{-1}_\delta,f^{-1}_\delta,w-n)$ by $(F_\delta,f_\delta,w+n)$.
\end{lem}

\begin{proof}
The statement follows from \cite[Lemma 5.1]{BT} since $\Psi_\delta$ is solution of differential equation $\dot z=f_\delta(z)-z=z(z+\delta)$.

The assumptions of \cite[Lemma 5.1]{BT} are satisfied because $f_\delta$ is $\theta$-stable for every $\theta\in(0,\pi/2-|\alpha|)$ (see \cite[Section 2, Example 1 (continued)]{BT}), and then the assumptions follows from \cite[proof of Lemma 5.3]{BT}.
\end{proof}

So, if $z=\Psi_\delta(w)$, then $w\in S^-(\theta)_R$ ($w\in S^+(\theta)_R$) can be considered as repelling (attracting) Fatou coordinate.

\section{Technical facts concerning exponential function}\label{sec:t}

In this section we state some elementary facts related to exponential function. The proofs will be given in Appendix \ref{app:A}.

\begin{lem}\label{lem:<}
For every $\ve,\tilde\ve\in(0,1)$, $\ve_1,\tilde\ve_1\greq0$ and $z, X, Y\in\C$
\begin{enumerate}
    \item\label{lit:<1}
       we have
          $$e^{|z|}(1+\ve)-1< e^{2|z|}-1+2\ve,$$
    \item\label{lit:<2}
       if $|X-1|< e^{\ve_1 |z|}-1+\ve$ and $|Y-1|< e^{\tilde\ve_1 |z|}-1+\tilde\ve$, then
          $$|XY-1|<e^{(2\ve_1+2\tilde\ve_1) |z|}-1+(2\ve+2\tilde\ve),$$
    \item\label{lit:<3}
       if $|X-1|< \ve$, then
          $$|Xe^{z}-1|< e^{2|z|}-1+2\ve.$$
 \end{enumerate}
\end{lem}

\begin{lem}\label{lem:1/2}
For every $\alpha\in(-\pi/2,\pi/2)$ and $\tilde m\greq 1$ there exist $\eta>0$ such that
   $$\Big|\sum_{k=\tilde m}^{m}e^{-k\delta}-1\Big|>\frac{1}{2}\sum_{k=\tilde m}^{m}\big|e^{-k\delta}-1\big|,$$
where $0<|\delta|<\eta$, $\alpha=\arg\delta$ and $m>\tilde m$.
\end{lem}

For $\alpha\in[-\pi/4,0)\cup(0,\pi/4]$ we define
\begin{equation}\label{eq:defw1}
   W_\alpha:=\big\{z\in\C:\arg z\in(-|\alpha|,|\alpha|)\;\wedge \;\re z>1/4\big\}.
\end{equation}
Next, for $\alpha\in(-\pi/2,-\pi/4)\cup(\pi/4,\pi/2)$, write
\begin{equation}\label{eq:defw2}
   W_\alpha:=\big\{z\in\C:\arg z\in(-|\alpha|,|\alpha|)\;\wedge \;\re z>(1/4)\cot|\alpha|\big\}.
\end{equation}

\begin{lem}\label{lem:w}
If $\alpha\in(-\pi/2,\pi/2)\sms\{0\}$ and $\alpha=\arg w$, then
$$\frac{1}{e^{w}-1}+\frac12\in W_\alpha.$$
\end{lem}

\begin{lem}\label{lem:e}
  If $\alpha\in(-\pi/2,\pi/2)$ and $\alpha=\arg\delta$, then for every $\ve\in(0,1)$, $w\in\R^-$ and $\tilde w\in\C$ such that $\tilde w\in B(w,\ve|w|\cos\alpha)$, we have
    $$\Big|\frac{e^{\tilde w\delta}-1}{e^{w\delta}-1}-1\Big|<\ve.$$
  Moreover, if $\tilde w\in B(w,\frac\ve2|w|\cos\alpha)$, then the above inequality holds after interchanging the roles of $w$ and $\tilde w$.
\end{lem}

We have (cf. (\ref{eq:Psi}))
\begin{equation}\label{eq:Psi'}
\Psi_\delta'(w)=\Big(\frac{\delta}{e^{-w\delta}-1}\Big)^2e^{-w\delta}=\Big(\frac{\delta}{e^{w\delta}-1}\Big)^2e^{w\delta} =\Big(\frac{\delta/2}{\sinh(w\delta/2)}\Big)^2,
\end{equation}
and then
\begin{equation}\label{eq:Psi'2}
\frac{\Psi_\delta'(\tilde w)}{\Psi_\delta'(w)}=\Big(\frac{e^{w\delta}-1}{e^{\tilde w\delta}-1}\Big)^2e^{(\tilde w-w)\delta} =\Big(\frac{\sinh(w\delta/2)}{\sinh(\tilde w\delta/2)}\Big)^2.
\end{equation}

The following lemma is a consequence of Lemma \ref{lem:e} and Lemma \ref{lem:<}.
\begin{lem}\label{lem:Psi'/}
  If $\alpha\in(-\pi/2,\pi/2)$ and $\alpha=\arg\delta$, then there exist $K(\alpha)$, such that for every $\ve\in(0,1)$, $w\in\R^-$ and $\tilde w\in\C$, where $\tilde w\in B(w,\ve|w|/K(\alpha))$, we have
    $$\Big|\frac{\Psi_\delta'(\tilde w)}{\Psi_\delta'(w)}-1\Big|<e^{\ve |w\delta|}-1+\ve.$$
  Moreover, the above inequality holds after interchanging the roles of $w$ and $\tilde w$ on the left-hand side.
\end{lem}

\section{Some properties of the Julia and the postcritical sets}\label{sec:P}

We begin with the following fact, which follows from the Fatou's flower theorem (see for example \cite[Lemma 8.2]{ADU}):
\begin{lem}\label{lem:fatouff}
For every $\theta$ there exists $r>0$ such that
$$\big(\mc J_0\cap B(0,r)\big)\subset \big(S^+(\theta,r)\cup\{0\}\big).$$
\end{lem}

\begin{figure}[!h]
\begin{center}
\includegraphics[width=6.2cm]{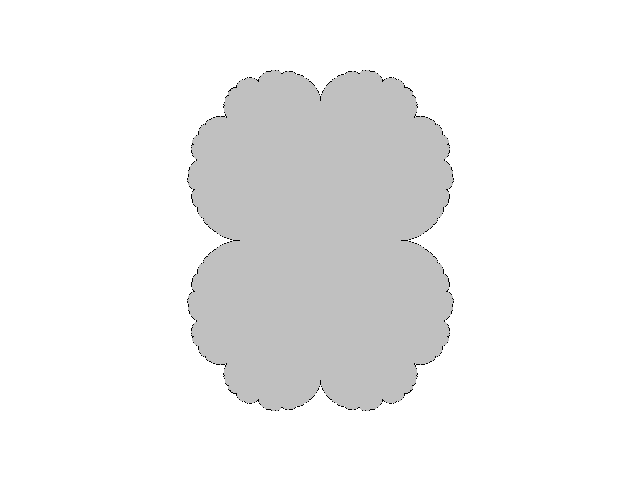} \includegraphics[width=6.2cm]{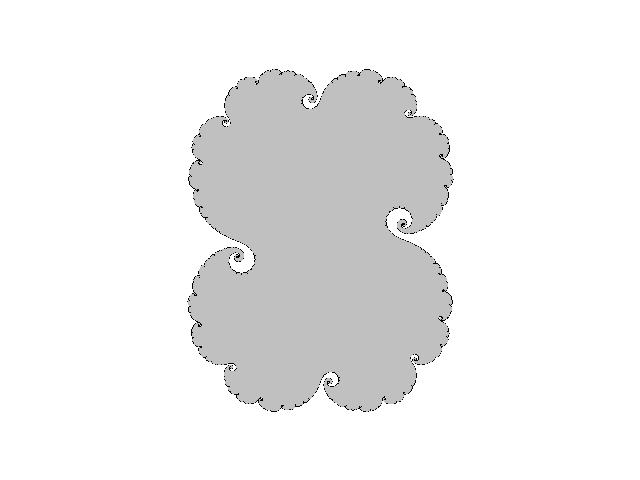}
\caption{The Julia sets for $\delta=0$ and $\delta=0.04+0.2i$ }
\end{center}
\end{figure}

Next, we conclude from \cite[Theorem 9.1]{Mii} (see also \cite{BT}), that
\begin{thm}\label{thm:Hmetric}
   If $\alpha\in(-\pi/2,\pi/2)$ and $\delta\rightarrow0$ where $\alpha=\arg\delta$, then
      $$\mc J_\delta\rightarrow\mc J_0,$$
   in the space of non-empty compact subsets of $\C$ equipped with the Hausdorff metric.
\end{thm}

The postcritical $P(f_\delta)$ set is defined as follows:
$$P(f_\delta)=\overline{\bigcup_{n\greq1}{f^n_\delta(c_\delta)}},$$
where $c_\delta=-1/2-\delta/2$ is the only critical point of $f_\delta$.

\begin{lem}\label{lem:kat}
   For every $\alpha\in(-\pi/2,\pi/2)$, $\theta>0$ and $R>0$ there exist $\eta>0$ and $r>0$ such that
  \begin{enumerate}
    \item\label{lit:kat1}
          $\big(\mc J_\delta\cap B(0,r)\big)\subset \big(\Psi_\delta(S^-(\theta)_R)\cup\{0\}\big),$
    \item\label{lit:kat2}
          $\big(P(f_\delta)\cap B(0,r)\big)\subset \big(\Psi_\delta(S^+(\theta)_R)\cup\{-\delta\}\big),$
 \end{enumerate}
   where $0<|\delta|<\eta$ and $\alpha=\arg\delta$ or $\delta=0$.
\end{lem}

\begin{proof}
Fix $\alpha\in(-\pi/2,\pi/2)$. We can assume that $\theta>0$ and $R>0$ are such that Lemma \ref{lem:trans} holds for some $\ve>0$ and $\eta>0$.

{\em Step 1.}
We know that the Julia set $\mc J_{0}$ approaches the fixed point $0$ tangentially to the horizontal direction (see Lemma \ref{lem:fatouff}), so we have
   $$\big(\mc J_{0}\cap B(0,r)\big)\subset \big(S^+(\theta/2,r)\cup\{0\}\big)\subset\big(\Psi_0(S^-(\theta/2)_{2R})\cup\{0\}\big),$$
for suitably chosen $r>0$.

Next, using Theorem \ref{thm:Hmetric} and Lemma \ref{lem:Psiunif}, we get
   $$\big(\mc J_{\delta}\cap (B(0,r)\sms \overline {B(0,r/2)})\big)\subset \big(S^+(\theta/2,r)\sms \overline {S^+(\theta/2,r/2)}\big)\subset\Psi_\delta(S^-(\theta)_R),$$
where $0<|\delta|<\eta$ and $\alpha=\arg\delta$.

Since $\Phi_\delta\circ f^{-1}_\delta\circ \Psi_\delta$ maps $S^-(\theta)_R$ into $S^-(\theta)_R$ (see Lemma \ref{lem:trans}), we see that $f^{-1}_\delta$ maps $\Psi_\delta(S^-(\theta)_R)$ into $\Psi_\delta(S^-(\theta)_R)$, so the first statement follows.

{\em Step 2.}
It is easy to see, that the sequence $f^n_0(-1/2)$, $n\greq1$ is monotone increasing, and is included in the interval $[-1/4,0]$, and converges to $0$. So, there exists $k\greq1$ such that $f_0^k(-1/2)\in\Psi_0(S^+(\theta)_R)$. Using Lemma \ref{lem:Psiunif} (possibly changing $\eta>0$) we get $f_\delta^k(c_p)\in\Psi_\delta(S^+(\theta)_R)$, where $0<|\delta|<\eta$.

Moreover, we can assume that $f_\delta^n(c_p)\notin B(0,r)$, for $1\leeq n<k$, $|\delta|<\eta$ and suitably chosen $r>0$.
As before, using Lemma \ref{lem:trans}, we see that $f_\delta$ maps $\Psi_\delta(S^+(\theta)_R)$ into $\Psi_\delta(S^+(\theta)_R)$. So the second statement follows from the fact that the sequence $f_\delta^n(c_p)$ tends to the attracting fixed point $-\delta$.
\end{proof}

\begin{cor}\label{cor:w}
For every $\alpha\in(-\pi/2,\pi/2)$ and $\theta\in(0,\pi/2-|\alpha|)$, there exist $\eta>0$ and $r>0$ such that
\begin{enumerate}
    \item\label{cit:w1}
          $\big(\mc J_\delta\cap B(0,r)\big)\subset \delta(W_{|\alpha|+\theta}-1/2)$,
    \item\label{cit:w2}
          $P(f_\delta)\subset (-\delta)(W_{|\alpha|+\theta}+1/2)$,
 \end{enumerate}
   where $0<|\delta|<\eta$ and $\alpha=\arg\delta$.
\end{cor}

\begin{proof}
For $\alpha=0$ the statement follows from Lemma \ref{lem:fatouff}, and the fact that for $\delta\in(0,1)$, we have $P(f_\delta)\subset(-1,-\delta]$.

Fix $\alpha\in(-\pi/2,\pi/2)\sms\{0\}$.
If $w\in S^-(\theta)$ and $|\alpha|+\theta<\pi/2$, then $-w\delta\in S^+(|\alpha|+\theta)$, so Lemma \ref{lem:w} gives us
$$\frac{1}{e^{-w\delta}-1}\in W_{|\alpha|+\theta}-\frac12.$$
Multiplying both sides by $\delta$, we obtain (cf. (\ref{eq:Psi}))
$$\Psi_\delta(w)\in\delta(W_{|\alpha|+\theta}-1/2).$$
Since $w\in S^-(\theta)$, the first statement follows from Lemma \ref{lem:kat} (\ref{lit:kat1}).

If $w\in S^+(\theta)$, then $-(-w\delta)\in S^+(|\alpha|+\theta)$.
So, using the formula
$$\Psi_\delta(w)=-\Psi_\delta(-w)-\delta,$$
we get
$$\Psi_\delta(w)\in(-\delta)(W_{|\alpha|+\theta}-1/2)-\delta=(-\delta)(W_{|\alpha|+\theta}+1/2).$$
Therefore Lemma \ref{lem:kat} (\ref{lit:kat2}) leads to
$$\big(P(f_\delta)\cap B(-\delta,r)\big)\subset (-\delta)(W_{|\alpha|+\theta}+1/2).$$
But we can assume that the trajectory of the critical point outside $B(-\delta,r)$ is close to the real line. Thus, the second statement follows from the fact that $\R^-\subset(-\delta)(W_{|\alpha|+\theta}+1/2)\cup B(-\delta,r)$, where $|\delta|<\eta$ for sufficiently small $\eta>0$.
\end{proof}

Let $S^+,S^-\subset\partial\D$ denote the closed upper and lower half-circle respectively. We conclude from Lemma \ref{lem:kat} (\ref{lit:kat1}) and Lemma \ref{lem:trans} that:
\begin{cor}\label{cor:gamma}
For every $\alpha\in(-\pi/2,\pi/2)$ and $\ve>0$ there exists $\eta>0$, such that
   $$\varphi_\delta(S^+)\subset\{z\in \C:\im z\greq-\ve\}\;\;\textrm{ and }\;\;\varphi_\delta(S^-)\subset\{z\in \C:\im z\leeq\ve\},$$
where $0<|\delta|<\eta$ and $\alpha=\arg\delta$.
\end{cor}

\begin{prop}\label{prop:CHm}
   For every $\alpha\in(-\pi/2,\pi/2)$ the convergence
   $$\varphi_\delta\rightarrow\varphi_0$$
   is uniform on the set $\partial\D$, where $\delta\rightarrow0$ and $\alpha=\arg\delta$.
\end{prop}

\begin{proof}
Fix $\alpha\in(-\pi/2,\pi/2)$.

{\em Step 1.}
Suppose, towards a contradiction, that there are no uniform convergence. So, there exists a sequence of points $s_n\in\partial\D$, and parameters $\delta_n\rightarrow0$, where $\arg(\delta_n)=\alpha$, such that
\begin{equation*}\label{eq:u1}
|\varphi_{\delta_n}(s_n)-\varphi_0(s_n)|>\ve,
\end{equation*}
for some $\ve>0$.
Passing to a subsequence if necessary, we can assume that $s_n\rightarrow s_0=e^{i\beta_0}$, and $\varphi_{\delta_n}(s_n)\rightarrow z_\infty$, where the arguments of $s_n$ are monotone.

Let $z_0:=\varphi_{0}(s_0)$. We see that $|z_\infty-z_0|\greq\ve$.

Since $\varphi_{\delta_n}(s_n)\in\mc J_\delta$ and $\mc J_\delta\rightarrow\mc J_0$ (see Theorem \ref{thm:Hmetric}) we conclude that $z_\infty\in \mc J_0$. So, there exists $s_\infty\in\partial\D$ such that $\varphi_0(s_\infty)=z_\infty$.

Let $l$ be be the (shortest) arc joining $s_0$ and $s_\infty$. Then we have
\begin{equation}\label{eq:u2}
f_0^m(\varphi_0(l))=\varphi_0(T^m(l))=\mc J_0,
\end{equation}
where $m$ is the smallest integer such that $2^m\beta\greq 2\pi$ and $\beta$ is the length of $l$.

{\em Step 2.}
We can find a sequence $s_n'\rightarrow s_0$ of periodic points of $T$, such that
\begin{equation}\label{eq:u3}
|\varphi_\delta(s_n')-z_0|<1/n,
\end{equation}
where $|\delta|<\eta_n$, for some $\eta_n>0$. Moreover we can assume that arguments of $s_n'$ are monotone and $s_n$, $s_n'$ tend to $s_0$ from the same side.

Next we take a subsequence $k_n$ of $\N$ such that $|\delta_{k_n}|<\eta_n$.
Let $l_n$ be the (shortest) arc joining $s_{k_n}$ and $s_n'$. Then $\varphi_{\delta_{k_n}}(l_n)$ is a curve whose endpoints are "close" to $z_0$ and $z_\infty$ (cf. (\ref{eq:u3})).
Thus, for $n$ sufficiently large, distance between $f_{\delta_{k_n}}^m(\varphi_{\delta_{k_n}}(l_n))= \varphi_{\delta_{k_n}}(T^m(l_n))$ and $\mc J_{\delta_{k_n}}$ is small in the Hausdorff metric (cf. (\ref{eq:u2}) and Theorem \ref{thm:Hmetric}).

Since $T^m(s_{k_n})$ and $T^m(s_n')$ tend to $T^m(s_0)$ from the same side, we see that for $n$ sufficiently large, the arcs $T^m(l_n)$ are included in $S^+$ or $S^-$. Therefore the curves $\varphi_{\delta_{k_n}}(T^m(l_n))$ are included in $\varphi_{\delta_{k_n}}(S^+)$ or $\varphi_{\delta_{k_n}}(S^-)$, and we have the required contradiction since $\varphi_{\delta_{k_n}}(S^\pm)$ cannot be close to $\mc J_{\delta_{k_n}}$ in the Hausdorff metric (see Corollary \ref{cor:gamma}).
\end{proof}

\section{Cylinders}\label{sec:cyl}

\subsection{}
Now we define \emph{cylinders} that we will use to describe partition of a neighborhood of the repelling/parabolic fixed point. Let
\begin{equation}\label{eq:z}
   s_n:=e^{\pi i/2^{n}} \;\;\textrm{ and }\;\;z_n(\delta):=\vp_\delta(s_n),
\end{equation}
where $\delta\in\{0\}\cup\mc M^+_0$, $n\in\N$. Put
   $$C_n^+:=\{e^{\pi i\beta}\in\partial\D:\beta\in(2^{-(n+1)},2^{-n}]\}.$$
So, $C_n^+$ is the arc between $s_{n+1}$ and $s_{n}$. Write $C_n^-:=\overline{C_n^+}$, and
then
   $$C_n:=C_n^+\cup C_n^-.$$
We see that $\bigcup_{n\in\N}C_n\cup\{1\}=\partial\D$.

Next, for $\delta\in\{0\}\cup\mc M^+_0$ and $n\in\N$, we define
   $$\mc C_n(\delta):=\vp_\delta(C_n).$$
Thus we have $\bigcup_{n\in\N}\mc C_n(\delta)\cup\{0\}=\mc J_\delta$ and $z_n(\delta)\in \mc C_n(\delta)$.
Note that $\vp_\delta(1)=0$ is repelling (or parabolic for $\delta=0$) fixed point.

Instead of the diameters of the cylinders, we use a quantity which will be called the \emph{size} of the cylinder and denoted by $|\mc C_n(\delta)|$. Write
\begin{equation}\label{eq:defsize}
|\mc C_n(\delta)|:=|z_n(\delta)-z_{n+1}(\delta)|.
\end{equation}

The sets of points which are "near" the fixed point $0$, is defined are follows:
\begin{equation*}
\mathcal{M}_N^*(\delta):=\bigcup_{n>N}\mc C_n(\delta), \;\;\;\;\;\mathcal{M}_N(\delta):=\bigcup_{n>N}\mc C_n(\delta)\cup\{0\},
\end{equation*}
where $N\in\N$. Note that the sets $\mathcal{M}_N(\delta)$ are closed. Let
$$\mathcal{B}_N(\delta):=\mc J_\delta\sms\mathcal{M}_N(\delta).$$
Hence $\mathcal{B}_N(\delta)$ is the set of points which are "far" from $0$.
Related subsets of $\partial\D$ will be denoted by $M_N$, $M_N^*$
and $B_N$. So we have $B_N\cup M_N=\partial\D$.

We know that $\mc J_0$ approaches the parabolic fixed point $0$ tangentially to the horizontal direction (cf. Lemma \ref{lem:fatouff}).
Thus, Proposition \ref{prop:CHm} leads to:

\begin{cor}\label{cor:M}
For every $\alpha\in(-\pi/2,\pi/2)$ and $r>0$ there exist $N\in\N$ and $\eta>0$ such that
   $$\mathcal{M}_N(\delta)\subset B(0,r),$$
where $0<|\delta|<\eta$ and $\alpha=\arg\delta$ or $\delta=0$.
\end{cor}

\subsection{}
Now we define related cylinders in the Fatou coordinates.

Let $\alpha\in(-\pi/2,\pi/2)$ and let $R=R(\alpha)>0$ be such that Lemma \ref{lem:trans} holds for $\theta=\pi/4-|\alpha|/2$, $\ve=1/10$ and some $\eta=\eta(\alpha)>0$.

Let $M=M(\alpha)$ be the smallest integer such that $\overline{\mc C_{M}(0)}\subset\Psi_0(S^-(\theta/2)_{R+1})$. We see from Lemma \ref{lem:fatouff} that such $M$ exists.
Changing $\eta>0$ if necessary, we conclude from Lemma \ref{lem:Psiunif} and Proposition \ref{prop:CHm}, that
$$\mc C_{M}(\delta)\subset\Phi_\delta(S^-(\theta/2)_{R+1})\subset\{z\in\C:\re z>0\},$$
where $\alpha=\arg\delta$ and $0<|\delta|<\eta$.

The function $\Phi_\delta$ can be defined on the set $\{z\in\C:\re z>0\}$ by taking branch of logarithm for which $\log(t)\in\R$ if $t\in\R^+$. So, let us define
  $$\hat C_M(\delta):=\Phi_\delta(\mc C_M(\delta)).$$
We have $\hat C_M(\delta)\subset S^-(\theta/2)_{R+1}$. Since the assertion of Lemma \ref{lem:trans} holds on $S^-(\theta)_R\supset S^-(\theta/2)_{R+1}$, the function $F^{-1}_\delta=\Phi_\delta\circ f^{-1}_\delta\circ \Psi_\delta$ is defined and univalent on $S^-(\theta)_R$. Thus, let us write
  $$\hat C_{M+n}(\delta):=F^{-n}_\delta(\hat C_M(\delta)),$$
where $n\greq1$. So, we see that $\hat C_{n}(\delta)\subset S^-(\theta)_R$ and $\mc C_{n}(\delta)=\Psi_\delta(\hat C_{n}(\delta))$, where $0<|\delta|<\eta$, $\alpha=\arg\delta$ or $\delta=0$ and $n\greq M$

Since the cylinder $\hat C_M(\delta)\subset S^-(\theta/2)_{R+1}$ is separated from the boundary of $S^-(\theta)_R$, the fact that $F^{-n}_\delta$ is univalent on $S^-(\theta)_R$ whereas $F^{-1}_\delta$ is close to the translation, leads to:
\begin{lem}\label{lem:c/c1}
There exists $K>1$ such that for every $\alpha\in(-\pi/2,\pi/2)$ there exists $\eta>0$ such that
          $$\diam\hat C_{n}(\delta)< K,$$
where $n\greq M$, $0<|\delta|<\eta$ and $\alpha=\arg\delta$ or $\delta=0$.
 \end{lem}

Now, using Lemma \ref{lem:trans}, we prove the following important fact:
\begin{lem}\label{lem:-n}
   For every $\alpha\in(-\pi/2,\pi/2)$ and $\ve>0$ there exist $N\in\N$ and $\eta>0$ such that if $w\in\hat C_n(\delta)$ then
     $$|w+n|<\ve n,$$
   where $n>N$, $0<|\delta|<\eta$ and $\alpha=\arg \delta$ or $\delta=0$.
\end{lem}

\begin{proof}
Fix $\ve>0$ and  $\alpha\in(-\pi/2,\pi/2)$. Let $R>0$ and $\eta>0$ be such that Lemma \ref{lem:trans} holds for $\theta=\pi/4-|\alpha|/2$ and $\ve/2$.

We can find $\tilde n\greq M$, such that the cylinder $\hat C_{\tilde n}(0)=\Phi_0(\mc C_{\tilde n}(0))$ is included in the set $S^-(\theta)_R$ (cf. Lemma \ref{lem:fatouff}).

Choosing suitable branch of the logarithm, we can assume that $\Phi_\delta$ converges uniformly to $\Phi_0$ on a neighborhood of $\mc C_{\tilde n}(0)$.
So, possibly changing $\eta>0$ and using Proposition \ref{prop:CHm}, we can get $\hat C_{\tilde n}(\delta)=\Phi_\delta(\mc C_{\tilde n}(\delta))\subset S^-(\theta)_R$, where $0<|\delta|<\eta$, $\alpha=\arg\delta$ or $\delta=0$.

Note that there exist $K>0$, such that for every $w\in\hat C_{\tilde n}(\delta)$ we have
   $$|w+\tilde n|<K,$$
where $0<|\delta|<\eta$, $\alpha=\arg\delta$ or $\delta=0$.
Thus, Lemma \ref{lem:trans} leads to
   $$|w+n|<K+\frac\ve2(n-\tilde n)<\ve n,$$
where $w\in\hat C_{n}(\delta)$, $n>N$ for suitably chosen $N\in\N$.
\end{proof}

\subsection{}
For $t\in\R$, let us define
$$V_t:=\{z\in\C:\re (z)<0 \wedge|\im (z) -t|<\pi\}.$$
The functions $\Psi_\delta$ are univalent on the sets $(1/\delta) V_t$ (see (\ref{eq:Psi})). Thus, using Lemma \ref{lem:c/c1}, we obtain:
\begin{lem}\label{lem:c/c2}
There exists $K>1$ such that for every $\alpha\in(-\pi/2,\pi/2)$ there exists $\eta>0$ such that
          $$\diam\mc C_n(\delta)< K |\mc C_n(\delta)|,$$
where $n\greq1$, $0<|\delta|<\eta$ and $\alpha=\arg\delta$.
 \end{lem}

Note that (\ref{eq:Psi}) and (\ref{eq:Psi'}) lead to
\begin{equation}\label{eq:Psi-n}
\Psi_\delta(-n)=\frac{\delta}{e^{n\delta}-1}\,\,\,\,\textrm{ and }\,\,\,\,
\Psi_\delta'(-n)=\Big(\frac{\delta}{e^{n\delta}-1}\Big)^2e^{n\delta}.
\end{equation}

\begin{lem}\label{lem:z}
  For every $\alpha\in(-\pi/2,\pi/2)$ and $\ve>0$ there exist $N\in\N$ and $\eta>0$ such that for every $z\in\mc C_n(\delta)$
    $$\Big|\frac{z}{\Psi_\delta(-n)}-1\Big|<e^{\ve n|\delta|}-1+\ve,$$
  where $n>N$, $0<|\delta|<\eta$ and $\alpha=\arg\delta$. Moreover the above inequality holds after interchanging the roles of $z$ and $\Psi_\delta(-n)$.
\end{lem}

\begin{proof}
Fix $\alpha\in(-\pi/2,\pi/2)$ and $\ve>0$.
Let $z\in\mc C_n(\delta)$ and let $z=\Psi_\delta(\tilde w)$ where $\tilde w\in\hat C_n(\delta)$, $\arg\delta=\alpha$. If $w=-n$ then (\ref{eq:Psi'2}) leads to
$$\Big|\frac{z}{\Psi_\delta(-n)}-1\Big|=\Big|\frac{\Psi_\delta(\tilde w)}{\Psi_\delta(w)}-1\Big| =\Big|\frac{e^{w\delta}-1}{e^{\tilde w\delta}-1}\,e^{(\tilde w-w)\delta}-1\Big|.$$

Let $\tilde\ve=(\ve/4)\cos\alpha$.
We can assume that $|\tilde w- w|<\tilde\ve n$ where $n>N$ and $0<|\delta|<\eta$ for sufficiently chosen $N\in\N$, $\eta>0$ (see Lemma \ref{lem:-n}). Thus, Lemma \ref{lem:e} combined with Lemma \ref{lem:<} (\ref{lit:<3}) leads to
$$\Big|\frac{e^{w\delta}-1}{e^{\tilde w\delta}-1}\,e^{(\tilde w-w)\delta}-1\Big| <e^{2\tilde\ve n|\delta|}-1+\frac\ve2.$$
Thus, the statement follows, since the roles of $z$ and $\Psi_\delta(-n)$ can be interchanged as in Lemma \ref{lem:e}.
\end{proof}

\begin{lem}\label{lem:size}
  For every $\alpha\in(-\pi/2,\pi/2)$ and $\ve>0$ there exist $N\in\N$ and $\eta>0$ such that
    $$(e^{\ve n|\delta|}+\ve)^{-1}<\frac{|\mc C_n(\delta)|}{|\Psi_\delta'(-n)|}<e^{\ve n|\delta|}+\ve,$$
  where $n>N$, $0<|\delta|<\eta$ and $\alpha=\arg\delta$.
\end{lem}

\begin{proof}
Fix $\ve>0$ and $\alpha\in(-\pi/2,\pi/2)$.
We know that $\diam \hat C_n(\delta)$ are uniformly bounded (see Lemma \ref{lem:c/c1}), whereas $\Psi_\delta$ is univalent on the sets $(1/\delta)V_t$. Thus, using Lemma \ref{lem:-n}, we can assume that the distortion of $\Psi_\delta$ on $\hat C_n(\delta)$ is as close to $1$ as we need, where $n>N$, $0<|\delta|<\eta$, $\alpha=\arg\delta$, for suitably chosen $N\in\N$ and $\eta>0$.

Thus, if $\Psi_\delta(\tilde w)=z\in\mc C_n(\delta)$, then Lemma \ref{lem:trans} and definition (\ref{eq:defsize}) lead to
$$\Big(1+\frac\ve4\Big)^{-1}<\frac{|\mc C_n(\delta)|}{|\Psi_\delta'(\tilde w)|}<1+\frac\ve4.$$

Next, possibly changing $N\in\N$ and $\eta>0$, we conclude from Lemma \ref{lem:-n} combined with Lemma \ref{lem:Psi'/}, that
\begin{equation*}
\Big(e^{(\ve/2) n|\delta|}+\frac\ve4\Big)^{-1}<\frac{|\Psi_\delta'(\tilde w)|}{|\Psi_\delta'(-n)|}<e^{(\ve/2) n|\delta|}+\frac\ve4,
\end{equation*}
where $0<|\delta|<\eta$, $n>N$. We multiply the above estimates and the statement follows (cf. Lemma \ref{lem:<} (\ref{lit:<1})).
\end{proof}

\begin{cor}\label{cor:mod}
For every $\alpha\in(-\pi/2,\pi/2)$ there exist $K>1$, $\tilde K>0$ and $\eta>0$ such that
\begin{enumerate}
    \item\label{cit:mod1}
       hyperbolic estimate: if $n|\delta|>1$, then
     \begin{equation*}
        \frac{1}{K}|\delta|^{2}e^{-Kn|\delta|}<|\mc C_n(\delta)|< K|\delta|^{2}e^{-\frac{1}{K}n|\delta|},
     \end{equation*}
    \item\label{cit:mod2}
       parabolic estimate: if $n|\delta|\leeq1$, then
     \begin{equation*}
        \frac{1}{Kn^2}<|\mc C_n(\delta)|<\frac{K}{n^2},
     \end{equation*}
    \item\label{cit:mod3}
       if $n\greq1$, then
     \begin{equation*}
        |\mc C_n(\delta)|<\frac{\tilde K}{n^2},
     \end{equation*}
 \end{enumerate}
where $0<|\delta|<\eta$ and $\alpha=\arg\delta$.
\end{cor}

\begin{proof}
The first two statements easily follows Lemma \ref{lem:size} and the rightmost expression of (\ref{eq:Psi-n}).

The third statement must be proven in the case $n|\delta|>1$.
For every $K>0$ we can find a constant $C>0$ such that
$$x^2<Ce^{\frac1Kx},$$
for every $x>1$.
So, we substitute $n|\delta|$ in place of $x$ and the third statement follows from the first.
\end{proof}

\begin{lem}\label{lem:distP}
For every $\alpha\in(-\pi/2,\pi/2)$ and $\ve>0$ there exist $N\in\N$ and $\eta>0$ such that if $z\in\mc C_n(\delta)$, then
$$\ve\dist(z,P(f_\delta))>|\mc C_n(\delta)|,$$
where $n>N$, $0<|\delta|<\eta$ and $\alpha=\arg\delta$.
\end{lem}

\begin{proof}
First let us note that for every $\beta\in(-\pi/2,\pi/2)$ there exists $K_{\beta}>0$ such that for every $z\in W_\beta-1/2$
\begin{equation}\label{eq:distW}
\dist( z,-W_\beta-1/2)>K^{-1}_{\beta}| z+1|,
\end{equation}
(cf. definitions (\ref{eq:defw1}), (\ref{eq:defw2})). For example we can take $K_{\beta}=\max(2,2\tan\beta)$.

Fix $\alpha\in(-\pi/2,\pi/2)$ and $\ve>0$ (small enough). Let $z\in\mathcal{M}_N(\delta)$ and $\theta\in(0,\pi/2-|\alpha|)$.
We can assume that $z\in \delta (W_{|\alpha|+\theta}-1/2)$ (see Corollaries \ref{cor:M} and \ref{cor:w}).
Since $P(f_\delta)\subset \delta(-W_{|\alpha|+\theta}-1/2)$, we conclude from (\ref{eq:distW}) that
\begin{equation}\label{eq:ddist}
\dist(z,P(f_\delta))>\dist(z,\delta(-W_{|\alpha|+\theta}-1/2))>K^{-1}_{|\alpha|+\theta}|z+\delta|.
\end{equation}

Using (\ref{eq:Psi-n}) we get
$$\Psi_\delta(-n)+\delta=\frac{\delta}{e^{n\delta}-1}+\delta=\frac{\delta e^{n\delta}}{e^{n\delta}-1}.$$
So, if $z\in\mc C_n(\delta)$, then we see from Lemma \ref{lem:z} that
\begin{equation}\label{eq:z+d}
|z+\delta|>\Big|\frac{\delta e^{n\delta}}{e^{n\delta}-1}\Big|-\Big|\frac{\delta}{e^{n\delta}-1}(e^{\ve n|\delta|}-1+\ve)\Big| >\frac12\Big|\frac{\delta e^{n\delta}}{e^{n\delta}-1}\Big|.
\end{equation}

Next, Lemma \ref{lem:size} combined with (\ref{eq:Psi-n}) leads to
\begin{equation}\label{eq:cn}
|\mc C_n(\delta)|<\Big|\frac{\delta}{e^{n\delta}-1}\Big|^2|e^{n\delta}|(e^{\ve n|\delta|}+\ve).
\end{equation}

Therefore, (\ref{eq:ddist}) and next (\ref{eq:z+d}), (\ref{eq:cn}) give us
$$\frac{|\mc C_n(\delta)|}{\dist(z,P(f_\delta))}<K_{|\alpha|+\theta}\frac{|\mc C_n(\delta)|}{|z+\delta|}<2K_{|\alpha|+\theta}\Big|\frac{\delta}{e^{n\delta}-1}\Big|(e^{\ve n|\delta|}+\ve).$$
Since there exists a constant $K_{|\alpha|+\theta}'>0$, such that
$$\Big|\frac{\delta}{e^{n\delta}-1}\Big|(e^{\ve n|\delta|}+\ve)<K_{|\alpha|+\theta}'\max\Big(\frac1n,|\delta|\Big),$$
the assertion follows.
\end{proof}

\section{Estimates of $(f_\delta^n)'$}

Now we give several important estimates concerning $(f_\delta^n)'$.
Let $w_{n}(\delta)\in\hat C_n(\delta)$ be the point such that $\Psi_\delta(w_{n}(\delta))=z_{n}(\delta)$ (cf. (\ref{eq:z}), (\ref{eq:defsize})).

\begin{lem}\label{lem:(f^j)'}
  For every $\alpha\in(-\pi/2,\pi/2)$ and $\ve>0$ there exist $N\in\N$ and $\eta>0$ such that for every $z\in\mc C_n(\delta)$
    $$\Big|\frac{1}{(f_\delta^k)'(z)}\frac{\Psi_\delta'(w_{n-k}(\delta))}{\Psi_\delta'(w_{n}(\delta))}-1\Big|<\ve,$$
  where $k\greq1$, $n-k\greq N$, $0<|\delta|<\eta$ and $\alpha=\arg\delta$. 
\end{lem}

\begin{proof}
Fix $\ve>0$ and $\alpha\in(-\pi/2,\pi/2)$. Let $\alpha=\arg\delta$.

Using Lemma \ref{lem:distP}, we can assume that distortion of $f^{-k}_\delta$ is close to $1$ on $\mc C_{n-k}(\delta)$, and then
\begin{equation}\label{eq:1/f'}
\Big|\frac{1}{(f_\delta^k)'(z)}\frac{z_{n-k}(\delta)-z_{n-k+1}(\delta)}{z_{n}(\delta)-z_{n+1}(\delta)}-1\Big|<\frac\ve4,
\end{equation}
where $n-k\greq N$, $k\greq1$, $0<|\delta|<\eta$, for suitably chosen $N\in\N$ and $\eta>0$.

We know that the functions $\Psi_\delta$ are univalent on the sets $(1/\delta)V_t$. Thus, possibly changing $N\in\N$ and $\eta>0$, we can assume that the distortion of $\Psi_\delta$ is close to $1$ on $\hat C_n(\delta)$, and then
$$\Big|\frac{z_{n-j}(\delta)-z_{n-j+1}(\delta)}{\Psi_\delta'(w_{n-j}(\delta))}-1\Big|<\frac\ve4,$$
where $j\greq0$, $n-j\greq N$ and $0<|\delta|<\eta$. Thus, the statement follows from (\ref{eq:1/f'}) combined with the above estimate (cf. Lemma \ref{lem:<} (\ref{lit:<2})).
\end{proof}

We have (cf. formula (\ref{eq:Psi'2}))
\begin{equation}\label{eq:Psi'/}
\frac{\Psi_\delta'(k-n)}{\Psi_\delta'(-n)}= \Big(\frac{e^{n\delta}-1}{e^{(n-k)\delta}-1}\Big)^2e^{-k\delta}=\Big(\frac{\sinh(\frac{n}{2}\delta)}{\sinh(\frac{n-k}{2}\delta)}\Big)^2.
\end{equation}

\begin{lem}\label{lem:(f^k)'}
  For every $\alpha\in(-\pi/2,\pi/2)$ and $\ve>0$ there exist $N\greq1$ and $\eta>0$ such that for every $z\in\mc C_n(\delta)$
    $$\Big|\frac{1}{(f_\delta^k)'(z)}\frac{\Psi_\delta'(k-n)}{\Psi_\delta'(-n)}-1\Big|<e^{\ve n|\delta|}-1+\ve,$$
  where $k\greq1$, $n-k\greq N$, $0<|\delta|<\eta$ and $\alpha=\arg\delta$. 
\end{lem}

\begin{proof}
Fix $\ve>0$ and $\alpha\in(-\pi/2,\pi/2)$. Let $\alpha=\arg\delta$ and let $\tilde\ve=\ve/K(\alpha)$, where $K(\alpha)$ is the constant from Lemma \ref{lem:Psi'/}.

Since we can assume that $|w_{n}(\delta)+n|<\tilde\ve n$ and $|w_{n-k}(\delta)+(n-k)|<\tilde\ve(n-k)$ (cf. Lemma \ref{lem:-n}), Lemma \ref{lem:Psi'/} gives us
$$\Big|\frac{\Psi_\delta'(k-n)}{\Psi_\delta'(w_{n-k}(\delta))}-1\Big|<e^{\ve (n-k)|\delta|}-1+\ve,\textrm{ and }
\Big|\frac{\Psi_\delta'(w_{n}(\delta))}{\Psi_\delta'(-n)}-1\Big|<e^{\ve n|\delta|}-1+\ve.$$
Thus, the assertion follows from Lemma \ref{lem:(f^j)'} and Lemma \ref{lem:<} (\ref{lit:<2}).
\end{proof}

We conclude from Lemma \ref{lem:fatouff} and Proposition \ref{prop:CHm} that:
\begin{cor}\label{cor:B'}
For every $\alpha\in(-\pi/2,\pi/2)$ and $N\greq1$ there exist $\kappa>0$ and $\eta>0$ such that if $z\in f_\delta^{-1}(\mathcal{B}_{N-1}(\delta))$, then
   $$|f'_\delta(z)|>1+\kappa,$$
where $0<|\delta|<\eta$ and $\alpha=\arg\delta$.
\end{cor}

Let $\mc C'_n(\delta)$ be the set which is placed symmetrically to $\mc C_n(\delta)$ with respect to the critical point $c_\delta=-1/2-\delta/2$. So we have $f_\delta(\mc C'_n(\delta))=\mc C_{n-1}(\delta)$, and then we see that $\mc C_n(\delta)\cup\mc C'_n(\delta)=f_\delta^{-1}(\mc C_{n-1}(\delta))$, where $n\greq1$.

\begin{lem}\label{lem:>1}
For every $\alpha\in(-\pi/2,\pi/2)$ and $N\in\N$ there exists $\eta>0$ such that for every $k\greq 1$ and $z\in\mc C_{N+k}(\delta)\cup\mc C'_{N+k}(\delta)$ we have
$$|(f^k_\delta)'(z)|>1,$$
where $0<|\delta|<\eta$ and $\alpha=\arg\delta$.
\end{lem}

\begin{proof}
Fix $\alpha\in(-\pi/2,\pi/2)$ and $\ve>0$ such that $\ve<1/2\cos\alpha$. We can assume that $N$ is large (cf. Corollary \ref{cor:B'}). So, let $N\in\N$ and $\eta>0$ be such that Lemma \ref{lem:(f^k)'} holds, and let $z\in\mc C_{N+k}(\delta)$ (if $z\in\mc C'_{N+k}(\delta)$, then we have the same estimates).

First, we assume that $(N+k)|\delta|\greq1$. There exists a constant $K_\alpha>0$ such that
   $$\Big|\sinh\Big(\frac{N+k}{2}\delta\Big)\Big|>K_\alpha e^{\frac{N+k}{2}\re\delta},$$
where $\alpha=\arg\delta$. So, changing $\eta>0$ if necessary, we get (cf. formula (\ref{eq:Psi'/}))
    $$\Big|\frac{\Psi_\delta'(-N)}{\Psi_\delta'(-N-k)}\Big|=\Big|\frac{\sinh(\frac{N+k}{2}\delta)}{\sinh(\frac{N}{2}\delta)}\Big|^2> e^{(N+k)\re\delta},$$
where $0<|\delta|<\eta$. So, if $\ve<1/2\cos\alpha$, then $(N+k)\re\delta>2\ve (N+k)|\delta|$ and $e^{(N+k)\re\delta}>e^{\ve (N+k)|\delta|}+\ve$.
So, the statement follows from Lemma \ref{lem:(f^k)'}.

Let $(N+k)|\delta|<1$. Changing $\eta>0$ if necessary, we can assume that $f'_\delta(z)>1$ where $z\in\mc C_{N+k}(z)$, $k\leeq N$ and $0<|\delta|<\eta$ (cf. Corollary \ref{cor:B'}). So, we will consider $2N<N+k<1/|\delta|$. Thus, using (\ref{eq:h3}), we obtain
$$\Big|\frac{\Psi_\delta'(-N)}{\Psi_\delta'(-N-k)}\Big|=\Big|\frac{\sinh(\frac{N+k}{2}\delta)}{\sinh(\frac{N}{2}\delta)}\Big|^2> \Big(\frac{\sin(1/2)}{\sin(1/4)}\Big)^2>2.$$
Because $2>e^{\ve (N+k)|\delta|}+\ve$, the statement follows from Lemma \ref{lem:(f^k)'}.
\end{proof}

\begin{cor}\label{cor:K}
For every $\alpha\in(-\pi/2,\pi/2)$ there exist $\kappa>0$ and $\eta>0$ such that for every $z\in\mc J_\delta$ and $n\greq1$
$$|(f_\delta^n)'(z)|>\kappa,$$
where $0<|\delta|<\eta$ and $\alpha=\arg\delta$.
\end{cor}

\begin{proof}
Fix $\alpha\in(-\pi/2,\pi/2)$ and $N\in\N$ (large enough). Let $\eta>0$ be such that Lemma \ref{lem:>1} and Corollary \ref{cor:B'} (for some $\kappa>0$) hold.

Thus we see that it is enough to consider trajectories such that all points $z, f_\delta(z),\ldots,f_\delta^{n}(z)$ are included in the set $\mc M_{N}(\delta)$. Since $\alpha=\arg\delta$ and $w_j(\delta)\in S^-(\theta)$ where $\theta=\pi/4-|\alpha|/2$, we conclude from (\ref{eq:h3}) that there exists $\kappa_\alpha>0$, such that
\begin{equation*}
\bigg|\frac{\Psi_\delta'(w_{n-k}(\delta))}{\Psi_\delta'(w_{n}(\delta))}\bigg|= \bigg|\frac{\sinh^2(w_{n}(\delta)\frac\delta2)}{\sinh^2(w_{n-k}(\delta)\frac\delta2)}\bigg|>
\kappa_\alpha\frac{\sinh^2(\re(w_{n}(\delta)\frac\delta2))}{\sinh^2(\re(w_{n-k}(\delta)\frac\delta2))}
\end{equation*}
We can assume that $\re(w_n(\delta)\delta/2)<\re(w_{n-k}(\delta)\delta/2)<0$ (cf. Lemma \ref{lem:trans}), thus the latter expression is bounded below by $\kappa_\alpha$, and the assertion follows from Lemma \ref{lem:(f^j)'}.
\end{proof}

Now we prove generalized version of \cite[formula (4.8)]{HZ}.

\begin{lem}\label{lem:n^2}
For every $\alpha\in(-\pi/2,\pi/2)$ there exists $\eta>0$ such that for every $N\greq1$ there exists $K(N)>0$ such that if $f^n_\delta(z)\in\mc B_N(\delta)$, $n\greq1$ then
\begin{equation}\label{eq:n^2}
|(f^n_\delta)'(z)|>K(N)n^2,
\end{equation}
where $0<|\delta|<\eta$ and $\alpha=\arg\delta$.
\end{lem}

\begin{proof}
Fix $\alpha\in(-\pi/2,\pi/2)$ and $\eta>0$ such that the segment joining $0$ and $\eta e^{i\alpha}$ is included in $\mc M^+_0\cup\{0\}$.

Let $L$ be a closed segment joining $\tilde\eta e^{i\alpha}$ to $\eta e^{i\alpha}$, where $0<\tilde \eta<\eta$. Then, because $I$ is included in the hyperbolic component $\mc M^+_0$, we can find $C>0$ and $\lambda>1$ such that
\begin{equation*}
|(f^n_\delta)'(z)|>C\lambda^n,
\end{equation*}
for every $z\in\mc J_\delta$ and $n\greq1$, provided $\delta\in I$. Thus, in order to prove lemma, it is enough to find $\tilde\eta>0$ which depends on $N\greq1$, such that (\ref{eq:n^2}) holds for $0<|\delta|<\tilde\eta$.

Fix $N\greq1$. Let $z_1,\ldots,z_m$ be a sequence of consecutive points from the trajectory $z,f_\delta(z)\ldots f_\delta^{n-1}(z)$, that are included in $f_\delta^{-1}(\mc M_{N-1}(\delta))$, and $f_\delta(z_m)\in f_\delta^{-1}(\mc B_{N-1}(\delta))$. Thus, we have $f_\delta(z_m)\in\mc C_N(\delta)\subset\mc B_N(\delta)$ and $z_1\in\mc C_{N+m}(\delta)\cup\mc C'_{N+m}(\delta)$. So, we can find $\tilde\eta>0$ (depending on $N$), such that
\begin{equation}\label{eq:m1}
|(f_\delta^m)'(z_1)|>1,
\end{equation}
where $0<|\delta|<\tilde\eta$ and $\alpha=\arg\delta$ (see Lemma \ref{lem:>1}).

Let $\tilde m$ be the length of the longest sequence $z'_1,\ldots,z'_{\tilde m}$ with properties as before. Because we can assume that the distortion of $f^{-\tilde m}_\delta$ is close to $1$ on a neighborhood of $\mc C_N(\delta)$ (cf. Lemma \ref{lem:distP}), there exists $K_1>0$ such that
\begin{equation}\label{eq:km}
|(f^{\tilde m}_\delta)'(z'_1)|> K_1\frac{|\mc C_{N}(\delta)|}{|\mc C_{N+\tilde m}(\delta)|}.
\end{equation}

Possibly changing $\tilde\eta>0$, we can assume that Corollary \ref{cor:B'} holds for some $\kappa>0$. Let $k_n(z)$ be the number of points from the trajectory, which are included in $f_\delta^{-1}(B_{N-1}(\delta))$. Then, using (\ref{eq:m1}) and (\ref{eq:km}), we conclude that
\begin{equation}\label{eq:kn}
|(f_\delta^n)'(z)|>K_1\frac{|\mc C_{N}(\delta)|}{|\mc C_{N+\tilde m}(\delta)|}(1+\kappa)^{k_n(z)},
\end{equation}

Obviously $|\mc C_{N+\tilde m}(\delta)|$ is bounded above, thus if $(1+\kappa)^{k_n(z)}>n^2$, then the assertion holds. Therefore, we will consider points $z\in\mc J_\delta$ for which $k_n(z)\leeq 2\log n/\log(1+\kappa)$.
Since we can assume that $n>\tilde n$, for some $\tilde n>1$ (depending on $N$), we can also assume that $k_n(z)\leeq n/2$.

So, there are at least $n/2$ points from the trajectory inside $f_\delta^{-1}(\mc M_{N-1}(\delta))$ and we conclude that $\tilde m\greq n/(2k_n(z))$. Thus, using Corollary \ref{cor:mod}, we get
\begin{equation*}\label{eq:tildem}
\frac{|\mc C_{N}(\delta)|}{|\mc C_{N+\tilde m}(\delta)|}>\frac{K_2(N)}{|\mc C_{N+\tilde m}(\delta)|}>K_3(N)(N+\tilde m)^2>K_4(N)\Big(\frac{n}{k_n(z)}\Big)^2.
\end{equation*}
Combining this with (\ref{eq:kn}), we obtain
$$|(f_\delta^n)'(z)|>K_5(N)(1+\kappa)^{k(z)}\Big(\frac{n}{k_n(z)}\Big)^2=K_5(N)\frac{(1+\kappa)^{k(z)}}{k_n^2(z)}\;n^2>K_6(N) n^2,$$
where $0<|\delta|<\tilde\eta$, and the proof is finished.
\end{proof}

\section{Invariant measures}\label{sec:miary}

The construction of the invariant measures which are equivalent to the conformal ones was carried out in \cite[Section 7]{J} and \cite[Section 6]{Ji}.
We used the method described in \cite{S}.

In this section we will not repeat the whole construction, but we will just define the partition and the jump transformation which are usually needed. Note that in our case the $f_\delta$-invariant and conformal measures were already denoted by $\mu_\delta$ and $\omega_\delta$ respectively (see Section \ref{sec:formalism}).

First we slightly modify the sets $\mc C_n^-(\delta)$. We will assume that $\mc C_n^-(\delta)$ contains $\overline z_{n+1}(\delta)$ instead of $\overline z_{n}(\delta)$. Thus, the sets $\{\mc C_n^-(\delta),\mc C_n^+(\delta)\}_{n\in\N}$ form a disjoint partition of $\mc J_\delta\sms\{0\}$ (in the construction $\{\mc C_{n-2}^-(\delta),\mc C_{n-2}^+(\delta)\}$ were denoted by $\mathscr B_n$, where $n\greq2$).

Define the \emph{jump transformation} $f_\delta^*:\mc J_\delta\sms\{0\}\rightarrow\mc J_\delta$ by
   $$f_\delta^*(z):=f_\delta^{n+2}(z)\;\;\textrm{ provided }\;\;z\in\mc C_n^-(\delta)\cup\mc C_n^+(\delta).$$
Note that for every $\mc C_n^\pm(\delta)$, the iteration $f_\delta^{n+2}$ (which maps $\mc C_n^\pm(\delta)$ injectively onto $\mc J_\delta$) as well as every inverse branch of $f_\delta^k$ defined on $\mc C_n^\pm(\delta)$, has uniformly bounded distortion (cf. Lemma \ref{lem:distP}).

Thus the construction of the unique (up to multiplicative constant) $f_\delta$-invariant measures $\mu_\delta$ equivalent to the conformal measure $\omega_\delta$ can be carried out. Note that the families denoted in the construction by $\mathscr D_n$, where $n\greq1$, consist of the sets $\bigcup_{k=n-1}^\infty \mc C_k^+(\delta)$ and $\bigcup_{k=n-1}^\infty \mc C_k^-(\delta)\cup\{0\}$.

We have $d(0)>1$, therefore $\mu_\delta$ is finite (see \cite[Theorem 9.10]{ADU}). So, we will assume that all the measures are normalized.
Next, we take
$\tilde{\mu_{\delta}}:=(\vp^{-1}_{\delta})_*\mu_{\delta}$, and $\tilde{\omega_{\delta}}:=(\vp^{-1}_{\delta})_*\omega_{\delta}$ (measures supported on
$\partial\D$).

Analogously as in \cite[Lemma 7.3]{J} and \cite[Lemma 6.4]{Ji} we obtain:

\begin{lem}\label{lem:mn}
For every $\alpha\in(-\pi/2,\pi/2)$ and $N\greq1$ there exist $D>1$, $\la(N)>1$ and $\eta>0$ such that
$$D^{-1}<\frac{d\mu_\delta}{d\omega_\delta}\Big|_{\mc B_N(\delta)}<\la(N),$$
where $0<|\delta|<\eta$, $\alpha=\arg\delta$ or $\delta=0$, and $D$ does not depend on $N$.
\end{lem}

\begin{lem}\label{lem:sumofmeasures}
   There exists constant $H>0$, such that for every $\alpha\in(-\pi/2,\pi/2)$ and $\varepsilon>0$ there exist $N\in\N$, $\eta>0$ such that
      $$(1-\varepsilon)H\sum_{k=n}^{\infty}\tilde{\omega_\delta}(C_k)
      <\tilde{\mu_\delta}(C_n)<(1+\varepsilon)H\sum_{k=n}^{\infty}\tilde{\omega_\delta}(C_k),$$
   where $n>N$, $0<|\delta|<\eta$ and $\alpha=\arg\delta$ or $\delta=0$.
\end{lem}

It follows from \cite[Theorem 11.2]{Mii} that:

\begin{prop}\label{prop:o}
For every $\alpha\in(-\pi/2,\pi/2)$ the measure $\tilde{\omega_0}$ is equal to weak* limit of $\tilde{\omega_\delta}$, where $\delta\rightarrow0$ and $\alpha=\arg\delta$.
\end{prop}

\begin{prop}\label{prop:mu}
For every $\alpha\in(-\pi/2,\pi/2)$ the measure $\tilde{\mu_0}$ is equal to weak* limit of $\tilde{\mu_\delta}$, where $\delta\rightarrow0$ and $\alpha=\arg\delta$.
\end{prop}

\begin{proof}
Fix $\alpha\in(-\pi/2,\pi/2)$. Let $\hat\mu_0$ be a week* limit of a sequence $\mu_{\delta_n}$ where $\delta_n\rightarrow0$ and $\alpha=\arg\delta_n$. Since $\vp_{\delta_n}$ converges uniformly to $\vp_0$ (see Proposition \ref{prop:CHm}), $\hat\mu_0$ is an $f_0$-invariant measure.

Next, we conclude from Proposition \ref{prop:o}, Lemma \ref{lem:mn}, and the uniqueness of the measure $\mu_0$, that there exists a constant $c\in(0,1]$ such that $\hat\mu_0=c\mu_0$ on the set $\mc J_0\sms\{0\}$. But we see from Lemma \ref{lem:sumofmeasures} that $c=1$, hence there are no atom at $0$ and $\hat\mu_0$=$\mu_0$.
\end{proof}

Now we are going to estimate $\tilde{\mu_\delta}(C_n)$.
We have $\tilde{\omega_\delta}(C_n)=\omega_\delta(\mc C_n(\delta))\asymp|\mc C_n(\delta)|^{\mc D(\delta)}$. Thus, Lemma \ref{lem:sumofmeasures} combined with Lemma \ref{lem:size} and (\ref{eq:Psi-n}) suggests that $\tilde{\mu_\delta}(C_n)$ can be estimated using the following expressions
   $$|\delta|^{2\mc D(\delta)}\int_{n}^{\infty}\Big|\frac{e^{x\delta}}{(e^{x\delta}-1)^{2}}\Big|^{\mc D(\delta)}|e^{\pm\ve x\delta}|dx =|\delta|^{2\mc D(\delta)-1}\int_{n|\delta|}^{\infty}\Big|\frac{e^{vs}}{(e^{vs}-1)^{2}}\Big|^{\mc D(\delta)}|e^{\pm\ve vs}|ds,$$
where $x\delta=xv|\delta|=vs$ (i.e. $x|\delta|=s$).
So, in order to state precise estimates (see Lemma \ref{lem:est}), let us define
\begin{equation}\label{eq:A}
   \Lambda_\ve^h(z):=\Big|\frac{e^{z}}{(e^{z}-1)^{2}}\Big|^h|e^{\ve z}| =\Big|\frac{1/4}{\sinh^2(z/2)}\Big|^h|e^{\ve z}|, 
\end{equation}
where $h>1$, $\ve\in[-1,1]$, and $\re z>0$.

Note that there exists $K>1$ (depending on $\alpha$) such that
\begin{equation*}\label{eq:ll}
 \begin{array}{ll}
   \Lambda_\ve^h(vt)<K e^{t(-h+\ve)\cos\alpha} &\textrm{ for }\;\;t\in(1,\infty),\\
   \Lambda_\ve^h(vt)<K t^{-2h} &\textrm{ for }\;\;t\in(0,1],
 \end{array}
\end{equation*}
where $v=e^{i\alpha}$. If $-h+\ve<0$ then there exists $\tilde K>1$ such that
\begin{equation}\label{eq:lll}
 \begin{array}{ll}
   \int_{t}^{\infty}\Lambda_\ve^h(vs)ds<\tilde K e^{t(-h+\ve)\cos\alpha} &\textrm{ for }\;\;t\in(1,\infty),\\
   \int_{t}^{\infty}\Lambda_\ve^h(vs)ds<\tilde K t^{-2h+1} &\textrm{ for }\;\;t\in(0,1].
 \end{array}
\end{equation}

\begin{lem}\label{lem:est}
   There exists $H_\mu>0$, and for every $\alpha\in(-\pi/2,\pi/2)$, $\varepsilon\in(0,1)$ there exist $N\in\N$, $\eta>0$ such that
     \begin{equation*}
        (1-\varepsilon)H_\mu\int^\infty_{n|\delta|}\Lambda_{-\varepsilon}^{\mc D(\delta)}(vs) ds<\frac{\tilde{\mu_\delta}( C_n)}{|\delta|^{2\mc D(\delta)-1}}< (1+\varepsilon)H_\mu\int^\infty_{n|\delta|}\Lambda_{\varepsilon}^{\mc D(\delta)}(vs) ds,
     \end{equation*}
   where $n>N$, $0<|\delta|<\eta$, $\alpha=\arg\delta$ and $v=e^{i\alpha}$.
\end{lem}

This Lemma can be proven analogously as in \cite[Lemma 6.5]{Ji}. If we do that, we will see that the constant $H_\mu$ does not depend
on the direction. But, that fact is important for us, so we will give an additional argument.

Let $\alpha\in(-\pi/2,\pi/2)$ and $v=e^{i\alpha}$. We see that $\Lambda_\ve^{\mc D(\delta)}(vs)$ is close to $s^{-2\mc D(\delta)}$ for small $s>0$. So we can get (cf. (\ref{eq:lll}))
 \begin{equation*}
   \lim_{\delta\rightarrow0}|\delta|^{2\mc D(\delta)-1}\int^\infty_{n|\delta|}\Lambda_{\pm\varepsilon}^{\mc D(\delta)}(vs) ds =\frac{n^{-2\mc D(\delta)+1}}{2\mc D(\delta)-1}.
\end{equation*}
On the other hand, for every $\alpha\in(-\pi/2,\pi/2)$ we have $\tilde{\mu_\delta}( C_n)\rightarrow\tilde{\mu_0}( C_n)$, where $\delta\rightarrow0$ and $\alpha=\arg\delta$. So, the constant $H_\mu$ cannot depend on $\alpha$.

\begin{cor}\label{cor:meas}
For every $\alpha\in(-\pi/2,\pi/2)$ there exist $K>1$ and $\eta>0$ such that
\begin{enumerate}
   \item\label{cit:meas1}
      if $n|\delta|>1$, then
    \begin{equation*}
       K^{-1}|\delta|^{2\mc D(\delta)-1}e^{-Kn|\delta|}< \tilde{\mu_\delta}( C_n)< K|\delta|^{2\mc D(\delta)-1}e^{-\frac1K n|\delta|},
    \end{equation*}
   \item\label{cit:meas2}
      if $n|\delta|\leeq1$, then
    \begin{equation*}
      K^{-1} n^{-2\mc D(\delta)+1}< \tilde{\mu_\delta}( C_n)< K n^{-2\mc D(\delta)+1},
    \end{equation*}
 \end{enumerate}
   where $n\greq1$, $0<|\delta|<\eta$ and $\alpha=\arg\delta$.
\end{cor}

\begin{cor}\label{cor:Kdelta}
For every $\alpha\in(-\pi/2,\pi/2)$ there exist $K>1$ and $\eta>0$ such that
$$K^{-1}|\delta|^{2\mc D(\delta)-2}<\tilde \mu_\delta(\mc M_{[1/|\delta|]})< K|\delta|^{2\mc D(\delta)-2},$$
where $0<|\delta|<\eta$ and $\alpha=\arg\delta$.
\end{cor}

A straightforward computation gives us the following lemma:

\begin{lem}\label{lem:K}
For every $\alpha\in(-\pi/2,\pi/2)$ there exist $K>1$, $\ve_0>0$ and $\eta>0$ such that
   $$\sum_{n=[1/|\delta|]+1}^{\infty}(e^{\ve n|\delta|}-1)\tilde{\mu_\delta}( C_n)< \ve K\tilde{\mu_\delta}(\mc M_{[1/|\delta|]}),$$
where $0<\ve<\ve_0$, $0<|\delta|<\eta$ and $\alpha=\arg\delta$.
\end{lem}

\begin{lem}\label{lem:mm}
   For every $\alpha\in(-\pi/2,\pi/2)$ and $\ve>0$ there exists $\eta>0$ and $N\in\N$ such that
     \begin{enumerate}
       \item\label{lit:mm1}
        $\sum_{n=[1/|\delta|]+1}^{\infty}\big|\tilde{\mu_\delta}( C_n)-H_\mu|\delta|^{2\mc D(\delta)-1}\int^{\infty}_{n|\delta|}\Lambda_0^{\mc D(\delta)}(vs)ds\big|< \ve |\delta|^{2\mc D(\delta)-2},$
       \item\label{lit:mm2}
         $\sum_{n=N+1}^{[1/|\delta|]}n|\delta|\big|\tilde{\mu_\delta}( C_n)-H_\mu|\delta|^{2\mc D(\delta)-1}\int^{\infty}_{n|\delta|}\Lambda_0^{\mc D(\delta)}(vs)ds\big|< \ve |\delta|^{2\mc D(\delta)-2},$
     \end{enumerate}
    where $0<|\delta|<\eta$, $\alpha=\arg\delta$ and $v=e^{i\alpha}$.
\end{lem}

\begin{proof}
Fix $\alpha\in(-\pi/2,\pi/2)$ and $\ve>0$ small enough.
Note that $\Lambda_{\ve}^{\mc D(\delta)}(z)-\Lambda_0^{\mc D(\delta)}(z)>\Lambda_{0}^{\mc D(\delta)}(z)-\Lambda_{-\ve}^{\mc D(\delta)}(z)$.
Thus, Lemma \ref{lem:est} leads to
\begin{multline}\label{eq:la}
\bigg|\tilde{\mu_\delta}( C_n)-H_\mu|\delta|^{2\mc D(\delta)-1}\int^{\infty}_{n|\delta|}\Lambda_0^{\mc D(\delta)}(vs)ds\bigg|\\
<H_\mu|\delta|^{2\mc D(\delta)-1}\bigg(\int^{\infty}_{n|\delta|}\Big(\Lambda_\ve^{\mc D(\delta)}(vs)-\Lambda_{0}^{\mc D(\delta)}(vs)\Big)ds+ \ve\int^{\infty}_{n|\delta|}\Lambda_\ve^{\mc D(\delta)}(vs)ds\bigg),
\end{multline}
where $0<|\delta|<\eta$, $\alpha=\arg\delta$, $n>N$, for suitably chosen $\eta>0$ and $N\in\N$.

We can assume that $\ve<1$, therefore Lemma \ref{lem:e} leads to
$$|e^{-x+\ve x}-e^{-x}|<\ve|e^{-x}-1|<\ve,$$
where $x>0$.
If $x=\re z$, then we have
\begin{equation*}\label{eq:d-d}
\Lambda_\ve^{\mc D(\delta)}(z)-\Lambda_{0}^{\mc D(\delta)}(z)
=(e^{-x+\ve x}-e^{-x})\Big|\frac{e^{z}}{(e^{z}-1)^{2}}\Big|^{\mc D(\delta)}|e^{z}|<\ve\Lambda_1^{\mc D(\delta)}(z).
\end{equation*}
Obviously $\Lambda_{\ve}^{\mc D(\delta)}(z)<\Lambda_1^{\mc D(\delta)}(z)$, so (\ref{eq:la}) gives us
\begin{equation}\label{eq:lam}
\bigg|\tilde{\mu_\delta}( C_n)-H_\mu|\delta|^{2\mc D(\delta)-1}\int^{\infty}_{n|\delta|}\Lambda_0^{\mc D(\delta)}(vs)ds\bigg|
<2\ve H_\mu|\delta|^{2\mc D(\delta)-1}\int^{\infty}_{n|\delta|}\Lambda_1^{\mc D(\delta)}(vs)ds.
\end{equation}

We can assume that $-\mc D(\delta)+1<0$. Thus, if $n|\delta|>1$, then using (\ref{eq:lll}) we see that there exist $K_1,K_2,K_3>0$ for which
\begin{equation*}
\sum_{n=[1/|\delta|]+1}^{\infty}\int^{\infty}_{n|\delta|}\Lambda_1^{\mc D(\delta)}(vs)ds<
K_1\sum_{n=[1/|\delta|]+1}^{\infty}e^{-n|\delta|K_2}
<K_3|\delta|^{-1}.
\end{equation*}
So the first statement follows from (\ref{eq:lam}).

If $n|\delta|\leeq 1$, then (\ref{eq:lll}) and the fact that $\mc D(\delta)<3/2$ lead to
\begin{equation*}
\sum_{n=N+1}^{[1/|\delta|]}n|\delta|\int^{\infty}_{n|\delta|}\Lambda_1^{\mc D(\delta)}(vs)ds< K_4\sum_{n=N+1}^{[1/|\delta|]}n|\delta|(n|\delta|)^{1-2\mc D(\delta)}<K_5|\delta|^{-1},
\end{equation*}
and the second statement follows from (\ref{eq:lam}).
\end{proof}

\section{The functions $\dot\vp_\delta$ and $\dot\psi_\delta$}\label{sec:f}

The main problem in the proof of the Theorem \ref{thm:direction2}, is the estimation of the following integral (cf. formula \ref{eq:d}):
\begin{equation}\label{eq:integral}
\int_{\partial\D}\frac{\partial}{\partial t}\log|f_{tv}'(\vp_{tv})|d\tilde\mu_{tv}.
\end{equation}
The integrand can be rewritten as follows:
\begin{equation*}\label{eq:f1}
\frac{\partial}{\partial t}\log|f'_{tv}(\vp_{tv})|=\re\Big(\frac{\frac{\partial}{\partial t}(f'_{tv}(\vp_{tv}))}{f'_{tv}(\vp_{tv})}\Big).
\end{equation*}
Let $\dot\vp_\delta:=\frac{\partial}{\partial \delta}\vp_\delta$, then we see that
\begin{equation*}
\frac{\partial}{\partial t}(f'_{tv}(\vp_{tv}))=\frac{\partial}{\partial t}(1+tv+2\vp_{tv})=v+2\frac{\partial}{\partial t}\vp_{tv}= v+2v\dot\vp_\delta\big|_{\delta=tv}.
\end{equation*}

So, we have to deal with the functions $\dot\vp_\delta$ and $1+2\dot\vp_\delta$.
In this section we derive two formulas for $\dot\vp_\delta$, and define the function $\dot\psi_\delta$ which is a "principal part" of $\dot\vp_\delta$.

Next we prove two propositions in which we estimate the functions $\dot\psi_\delta$ and $1+2\dot\psi_\delta$.
These results will allow us to estimate integral (\ref{eq:integral}) restricted to a set $ M_N$, which has a decisive influence on (\ref{eq:integral}) and consequently on $\mc D'(\delta)$ (see Section \ref{sec:mn}).

\subsection{}
We know that the function $\vp_\delta$ conjugates $T(s)=s^2$ to $f_\delta(z)=(1+\delta)z+z^2$, i.e. $f_\delta\circ\vp_\delta=\vp_\delta\circ T$.
If $\tau_\delta(z)=z+(1+\delta)/2$, then we have (cf. (\ref{eq:conj}))
\begin{equation*}
 \begin{array}{ll}
   \tau_\delta\circ \vp_\delta\circ T=\tau_\delta\circ f_\delta\circ\vp_\delta,\\
   \tau_\delta\circ f_\delta=p_{-\delta^2/4}\circ\tau_\delta.
 \end{array}
\end{equation*}
So, we conclude that $\tau_\delta\circ\vp_\delta\circ T=p_{-\delta^2/4}\circ\tau_\delta\circ\vp_\delta$, hence
$$\vp_\delta(s^2)+\frac{1+\delta}{2}=\Big(\vp_\delta(s)+\frac{1+\delta}{2}\Big)^2+\frac14-\frac{\delta^2}{4}.$$
Differentiating both sides with respect to $\delta$, we  get
\begin{equation*}\label{eq:dot1}
\dot\vp_\delta(s^2)+\frac12=2\Big(\vp_\delta(s)+\frac{1+\delta}{2}\Big)\Big(\dot\vp_\delta(s)+\frac12\Big)-\frac\delta2.
\end{equation*}
We see that
$$\dot\vp_\delta(s)+\frac12=\frac{\delta/2}{2\vp_\delta(s)+1+\delta}+\frac{\dot\vp_\delta(s^2)+1/2}{2\vp_\delta(s)+1+\delta},$$
and then
\begin{equation*}
\dot\vp_\delta(s)+\frac12=\frac{\delta/2}{f_\delta'(\vp_\delta(s))}+\frac{\dot\vp_\delta(s^2)+1/2}{f_\delta'(\vp_\delta(s))}.
\end{equation*}

Next, replacing $s$ by $s^2$, $s^4$,..., $s^{2^{m-1}}$, we obtain
\begin{multline*}\label{eq:sumass1}
\dot\vp_\delta(s)+\frac12=\sum_{k=0}^{m-1}
\frac{\delta/2}{f_\delta'(\vp_\delta(s))\cdot f_\delta'(\vp_\delta(s^2))\cdot...\cdot f_\delta'(\vp_\delta(s^{2^{k}}))}+\\
+\frac{\dot\vp_\delta(s^{2^m})+1/2}{f_\delta'(\vp_\delta(s))\cdot f_\delta'(\vp_\delta(s^2))\cdot...\cdot f_\delta'(\vp_\delta(s^{2^{m-1}}))}.
\end{multline*}
Thus, for every $m\greq1$ we have
\begin{equation}\label{eq:sumas1}
\dot\vp_\delta(s)=-\frac12+\sum_{k=1}^{m}
\frac{\delta/2}{(f_\delta^k)'(\vp_\delta(s))}+\frac{\dot\vp_\delta(T^m(s))+1/2}{(f_\delta^m)'(\vp_\delta(s))}.
\end{equation}

If $\vp_\delta(s)=z\in\mc C_n(\delta)$, $n\in\N$, then we define
\begin{equation}\label{eq:psi1}
\dot\psi_\delta(z):=-\frac12+\sum_{k=1}^{n}
\frac{\delta/2}{(f_\delta^k)'(z)}.
\end{equation}
Hence $\dot\psi_\delta(z)$ is defined on the set $\mc M_0^*(\delta)$, and is a "principal part" of $\dot\varphi_\delta(s)$ close to the fixed point $0$, provided $z=\vp_\delta(s)$.

\subsection{}
Now we give another formulas for $\dot\vp_\delta$ and for $\dot\psi_\delta$.
Of course we have
$$\vp_\delta(s^2)=(1+\delta)\vp_\delta(s)+\vp_\delta^2(s).$$
So, as before, differentiating both sides we can get
$$\dot\vp_\delta(s)=\frac{-\vp_\delta(s)}{(1+\delta)+2\vp_\delta(s)}+\frac{\dot\vp_\delta(s^2)}{(1+\delta)+2\vp_\delta(s)}.$$
Next, repeating the above procedure, we obtain
\begin{equation*}\label{eq:sumas}
\dot\vp_\delta(s)=-\sum_{k=1}^{m}
\frac{f_\delta^{k-1}(\vp_\delta(s))}{(f_\delta^k)'(\vp_\delta(s))}+\frac{\dot\vp_\delta(T^m(s))}{(f_\delta^m)'(\vp_\delta(s))}.
\end{equation*}

If $\vp_\delta(s)=z\in\mc C_n(\delta)$, then the above formula combined with (\ref{eq:sumas1}) and (\ref{eq:psi1}) leads to
\begin{equation}\label{eq:psi}
\dot\psi_\delta(z)=-\sum_{k=1}^{n}
\frac{f_\delta^{k-1}(z)}{(f_\delta^k)'(z)}-\frac{1/2}{(f_\delta^n)'(z)}.
\end{equation}

\subsection{}
Now we are going to define the function $\Gamma$, which gives us an approximation of $\dot\psi_\delta$.
First, let us write
\begin{equation}\label{eq:defg}
g(z):=e^{-z}-1+z.
\end{equation}
Fix $\alpha\in(-\pi/2,\pi/2)$ and let $v=e^{i\alpha}$. We have $|g(tv)|\rightarrow\infty$, where $t\rightarrow\infty$. Moreover it is easy to see that $e^{-tv}-1\neq-tv$ where $t>0$, therefore the function $g(z)$ does not vanish in the half-plane $\re z>0$, whereas $g(0)=0$.

Let $E:=\{2k\pi i:k\in \Z\sms\{0\}\}$.
The function $\Gamma:\C\sms E\rightarrow\C$ is defined as
\begin{equation}\label{eq:f}
   \Gamma(z)=\frac12\Big(-1+\frac{\sinh (z)-z}{\cosh (z)-1}\Big)=\frac12\Big(\frac{-g(z)}{\cosh(z)-1}\Big)=\frac{e^{z}-ze^z-1}{(e^{z}-1)^2},
\end{equation}
where $z\neq0$, and $\Gamma(0)=-1/2$.
Note that close to $0$ we have
\begin{equation}\label{eq:2G+1}
1+2\Gamma(z)=\frac{\sinh (z)-z}{\cosh (z)-1}=\frac{z}{3}+O(z^3).
\end{equation}
So, we see that $\Gamma$ is continuous at 0. Moreover $\Gamma(z)\neq0$ if $\re z>0$, and $\lim_{t\rightarrow\infty}\Gamma(te^{i\alpha})=0$, where $\alpha\in(-\pi/2,\pi/2)$. Thus the function $\Gamma$ is bounded on each ray $\mc R(\alpha)=\{z\in\C^*:\alpha=\arg z\}$.

\subsection{}
Now we prove that $\Gamma(n\delta)$ is a good approximation of $\dot\psi_\delta(z)$ where $z\in\mc C_n(\delta)$ (see Proposition \ref{prop:psi}). But this is not enough for our purposes and we will also need another estimate.

In order to prove Theorem \ref{thm:direction2}, we have to show that the integral (\ref{eq:integral}) restricted to a set $M_N$, after dividing by $t^{2\mc D(0)-2}$, tends to a constant (see Proposition \ref{prop:MN}).

But, if $t\rightarrow0^+$ (i.e. $\delta\rightarrow0$), then $t^{2\mc D(0)-2}\rightarrow0$, so we expect that the integral also tends to $0$. On the other hand we have
$$\tilde{\mu_\delta}(M_N)\rightarrow\tilde{\mu_0}(M_N)>0.$$
Thus the integrand must be estimated very precisely. In particular, it is not enough to show that absolute value of the difference between $\dot\psi_\delta(z)$ and $\Gamma(n\delta)$ is less than a small $\ve$ (cf. Proposition \ref{prop:psi}).

It would be enough to show that $1+2\dot\psi_\delta(z)$ divided by $1+2\Gamma(n\delta)$ is close to $1$. But $1+2\Gamma(n\delta)$  vanishes at some points, therefore we will be able to prove it only under assumption that $n|\delta|\leeq2$ (see Proposition \ref{prop:psi2}).

However, in the case $n|\delta|>2$ (or in the case $n|\delta|>1$), the estimate from the following proposition will be enough for us.

\begin{prop}\label{prop:psi}
   For every $\alpha\in(-\pi/2,\pi/2)$ and $\varepsilon>0$ there exist $\eta>0$, $N\in\N$ such that if $z\in\mc C_n(\delta)$, then
 \begin{equation*}
    \Big|\frac{\dot\psi_\delta(z)}{\Gamma(n\delta)}-1\Big|<e^{\ve n|\delta|}-1+\ve,
 \end{equation*}
   where $n>N$, $0<|\delta|<\eta$ and $\alpha=\arg\delta$.
\end{prop}

\begin{proof}
Let $z\in \mc C_n(\delta)$, and let us write (cf. definition (\ref{eq:Psi}) and (\ref{eq:Psi'/}))
\begin{align*}
   a^k(z)&:=-\frac{f_\delta^{k-1}(z)}{(f_\delta^k)'(z)}\;\;\textrm{ and},\\
   \tilde a_n^k&:=-\Psi_\delta(k-n)\frac{\Psi_\delta'(-n)}{\Psi_\delta'(k-n)} =-\frac{\delta}{e^{(n-k)\delta}-1}\Big(\frac{e^{(n-k)\delta}-1}{e^{n\delta}-1}\Big)^2e^{k\delta}\\
   &=\delta\frac{e^{n\delta}}{(e^{n\delta}-1)^2}\big(e^{(k-n)\delta}-1\big)=\frac{\delta/2}{\cosh(z)-1}\big(e^{(k-n)\delta}-1\big).
 \end{align*}
Using formula (\ref{eq:psi}), we see that
\begin{equation}\label{eq:sa^k}
\dot\psi_\delta(z)=\sum_{k=1}^{n}a^k(z)-\frac{1/2}{(f_\delta^n)'(z)}.
\end{equation}

{\em Step 1.}
Fix $\alpha\in(-\pi/2,\pi/2)$ and $\ve>0$.
We can find $\tilde n\greq1$ and $\eta>0$ such that
$$\Big|\frac{\Psi_\delta(k-1-n)}{\Psi_\delta(k-n)}-1\Big|= \Big|\frac{e^{(n-k)\delta}-1}{e^{(n-k+1)\delta}-1}-1\Big|= \Big|\frac{e^{-(n-k)\delta}-1}{e^{-(n-k+1)\delta}-1}\,e^{-\delta}-1\Big|<\frac{\ve}{16},$$
where $0<|\delta|<\eta$ and $n-k\greq\tilde n$ (see Lemma \ref{lem:e} and Lemma \ref{lem:<} (\ref{lit:<3})).
Since $f_\delta^{k-1}(z)\in \mc C_{n-k+1}(\delta)$, Lemma \ref{lem:z} combined with the above, gives us
\begin{equation*}\label{eq:k-1}
\Big|\frac{f_\delta^{k-1}(z)}{\Psi_\delta(k-n)}-1\Big|<e^{\ve n|\delta|/4}-1+\frac\ve4.
\end{equation*}
Thus, Lemma \ref{lem:(f^k)'} leads to
$$\Big|\frac{a^k(z)}{\tilde a^k_n}-1\Big|<e^{\ve n|\delta|}-1+\ve,$$
where $n-k\greq\tilde n$, $0<|\delta|<\eta$, and $\alpha=\arg\delta$, for suitably chosen $\tilde n\greq1$ and $\eta>0$.
So, we obtain
\begin{equation}\label{eq:sss}
\Big|\sum_{k=1}^{n-\tilde n} a^k(z)-\sum_{k=1}^{n-\tilde n}\tilde a_n^k\Big|< (e^{\ve n|\delta|}-1+\ve)\sum_{k=1}^{n-\tilde n}\big|\tilde a_n^k\big|.
\end{equation}

{\em Step 2.}
We have
\begin{equation}\label{eq:akn}
\sum_{k=1}^{n-\tilde n}\tilde a_n^k
=\frac{\delta/2}{\cosh(z)-1}\sum_{k=1}^{n-\tilde n}\big(e^{(k-n)\delta}-1\big).
\end{equation}
So, we will deal with $\sum_{k=1}^{n-\tilde n}(e^{(k-n)\delta}-1)$. Note that Lemma \ref{lem:1/2} for $\tilde m=\tilde n$ and $m=n-1$ leads to
  \begin{equation}\label{eq:||}
      \Big|\sum_{k=1}^{n-\tilde n}\big(e^{(k-n)\delta}-1\big)\Big|>\frac{1}{2}\sum_{k=1}^{n-\tilde n}\big|e^{(k-n)\delta}-1\big|.
  \end{equation}

We know that $(e^{-m\delta}-1)/(e^{-(m+1)\delta}-1)$ is close to $1$, where $m>\tilde n$ (see Lemma \ref{lem:e}), so using (\ref{eq:||}) we obtain
\begin{equation}\label{eq:sc}
\Big|\sum_{k=1}^{n-\tilde n}\big(e^{(k-n)\delta}-1\big)-\int_{0}^{n-\tilde n}\big(e^{(x-n)\delta}-1\big)dx\Big| \leeq\ve\sum_{k=1}^{n-\tilde n}\big|e^{(k-n)\delta}-1\big|.
\end{equation}

We have (cf. definition (\ref{eq:defg}))
\begin{equation}\label{eq:int}
\int_{0}^{n-\tilde n}\big(e^{(x-n)\delta}-1\big)dx=\frac1\delta\big(e^{-\tilde n\delta}-e^{-n\delta}\big)-(n-\tilde n)=\frac1\delta\big(g(\tilde n\delta)-g(n\delta)\big).
\end{equation}
For $n>N$, where $N$ is large enough, $|g(\tilde n\delta)|$ is small with respect to $|g(n\delta)|$, thus we can get
\begin{equation}\label{eq:g}
   |g(\tilde n\delta)|<\ve|g(n\delta)|.
\end{equation}
Thus, (\ref{eq:sc}) and (\ref{eq:int}) lead to
\begin{equation}\label{eq:gg}
\Big|\sum_{k=1}^{n-\tilde n}\big(e^{(k-n)\delta}-1\big)+\frac1\delta g(n\delta)\Big| \leeq
\ve\Big|\frac1\delta g(n\delta)\Big|+\ve\sum_{k=1}^{n-\tilde n}\big|e^{(k-n)\delta}-1\big|.
\end{equation}
Equality (\ref{eq:int}) combined with (\ref{eq:g}) and (\ref{eq:sc}) gives us
$$\Big|\frac1\delta g(n\delta)\Big|<\frac{1}{1-\ve}\Big|\sum_{k=1}^{n-\tilde n}e^{(k-n)\delta}-1\Big|+\frac{\ve}{1-\ve}\sum_{k=1}^{n-\tilde n}\big|e^{(k-n)\delta}-1\big|.$$
So, using (\ref{eq:||}), we conclude from (\ref{eq:gg}) that
\begin{equation}\label{eq:e-1}
\Big|\sum_{k=1}^{n-\tilde n}\big(e^{(k-n)\delta}-1\big)+\frac1\delta g(n\delta)\Big| \leeq4\ve\Big|\sum_{k=1}^{n-\tilde n}e^{(k-n)\delta}-1\Big|.
\end{equation}

{\em Step 3.}
We have
\begin{equation*}
\frac{\delta/2}{\cosh(z)-1}\Big(-\frac1\delta g(n\delta)\Big)=\Gamma(n\delta).
\end{equation*}
Thus (\ref{eq:e-1}) gives us (cf. (\ref{eq:akn}))
$$\Big|\sum_{k=1}^{n-\tilde n}\tilde a_n^k-\Gamma(n\delta)\Big|< 4\ve\Big|\sum_{k=1}^{n-\tilde n}\tilde a_n^k\Big|.$$
So, using (\ref{eq:sss}) and (\ref{eq:||}) we obtain
\begin{multline}\label{eq:ssg}
\Big|\sum_{k=1}^{n-\tilde n} a^k(z)-\Gamma(n\delta)\Big|<2 \big(e^{\ve n|\delta|}-1+\ve\big)\Big|\sum_{k=1}^{n-\tilde n}\tilde a_n^k\Big|+4\ve\Big|\sum_{k=1}^{n-\tilde n}\tilde a_n^k\Big|\\
<(1-4\ve)^{-1}\big(2(e^{\ve n|\delta|}-1)+6\ve\big)\big|\Gamma(n\delta)\big|<(e^{3\ve n|\delta|}-1+7\ve)\big|\Gamma(n\delta)\big|.
\end{multline}

{\em Step 4.}
In order to finish the proof, we have to estimate (cf. (\ref{eq:sa^k})):
\begin{multline}\label{eq:a^k}
\bigg|\sum_{k=n-\tilde n+1}^n a^k(z)-\frac{1/2}{(f_\delta^n)'(z)}\bigg|\\
= \bigg|\frac{-1}{(f_\delta^{n-\tilde n})'(z)}\bigg(\sum_{k=1}^{\tilde n}\frac{f_\delta^{k-1}(z)}{(f_\delta^{k})'(f_\delta^{n-\tilde n}(z))}+\frac{1/2}{(f_\delta^{\tilde n})'(f_\delta^{n-\tilde n}(z))}\bigg)\bigg|.
\end{multline}
Note that absolute value of the expression in bracket can be bounded above by a constant $K_1(\tilde n)$ (cf. Corollary \ref{cor:K}). So, we will deal with $|(f_\delta^{n-\tilde n})'(z)|^{-1}$.

Let us assume that $n\re\delta\leeq2$. If $z\in\mc C_n(\delta)$ then $f_\delta^{n-\tilde n}(z)\in\mc B_{\tilde n}(\delta)$, therefore Lemma \ref{lem:n^2} gives us
\begin{equation}\label{eq:a^k2}
\bigg|\sum_{k=n-\tilde n+1}^n a^k(z)-\frac{1/2}{(f_\delta^n)'(z)}\bigg|<\frac{K_2(\tilde n)}{(n-\tilde n)^2}K_1(\tilde n)=\frac{K_3(\tilde n)}{(n-\tilde n)^2}.
\end{equation}
We have $n\re\delta\leeq2$ and $\alpha=\arg\delta$ so there exists a constant $\gamma(\alpha)>0$ such that $\gamma(\alpha)<\Gamma(n\delta)$.
If $n>N$ (for sufficiently large $N$) then we can get $K_3(\tilde n)(n-\tilde n)^{-2}<\ve\gamma(\alpha)$, so the statement follows from (\ref{eq:ssg}) and (\ref{eq:a^k2}).

Now let us assume that $n\re\delta>2$. Lemma \ref{lem:(f^k)'} and (\ref{eq:Psi'/}) lead to
\begin{equation*}
\frac{1}{|(f_\delta^{n-\tilde n})'(z)|}<e^{\ve n|\delta|}\Big|\Big(\frac{e^{\tilde n\delta}-1}{e^{n\delta}-1}\Big)^2e^{(n-\tilde n)\delta}\Big|<e^{\ve n|\delta|}\frac{2\tilde n^2|\delta|^2}{|e^{n\delta}-1|^2}e^{n\re\delta}.
\end{equation*}
Since $n\re\delta>2$, we have $3|e^{n\delta}-n\delta \,e^{n\delta}-1|>n|\delta|e^{n\re\delta}$, therefore
\begin{multline*}
2\tilde n^2|\delta|^2\frac{e^{n\re\delta}}{|e^{n\delta}-1|^2}e^{\ve n|\delta|}<\frac{\ve n|\delta|}{3}\frac{ e^{n\re\delta}}{|e^{n\delta}-1|^2}e^{\ve n|\delta|}\\
<\ve e^{\ve n|\delta|}\Big|\frac{e^{n\delta}-n\delta \,e^{n\delta}-1}{(e^{n\delta}-1)^2}\Big|=\ve e^{\ve n|\delta|}\big|\Gamma(n\delta)\big|,
\end{multline*}
where $0<|\delta|<\eta$ and $n>N$, for sufficiently chosen $\eta>0$ and $N\in\N$.
Thus, the expression (\ref{eq:a^k}) can be estimated by $\ve e^{\ve n|\delta|}|\Gamma(n\delta)|$, so the statement follows from (\ref{eq:ssg}).
\end{proof}

\begin{prop}\label{prop:psi2}
   For every $\alpha\in(-\pi/2,\pi/2)$ and $\varepsilon>0$ there exist $\eta>0$, $N\in\N$ such that if $z\in\mc C_n(\delta)$, then
 \begin{equation*}
    \Big|\frac{1+2\,\dot\psi_\delta(z)}{1+2\,\Gamma(n\delta)}-1\Big|<\ve,
 \end{equation*}
   where, $0<|\delta|<\eta$, $\alpha=\arg\delta$ and $N<n\leeq2/|\delta|$.
\end{prop}

\begin{proof}{\em Step 1.}
Fix $\alpha\in(-\pi/2,\pi/2)$.
Let $z\in \mc C_n(\delta)$. Let us write (cf. (\ref{eq:Psi'/}))
\begin{align}\label{eq:11}
   b^k(z)&:=\frac{\delta}{(f_\delta^k)'(z)}\;\;\textrm{ and},\nonumber\\
   \tilde b_n^k&:=\delta\frac{\Psi_\delta'(-n)}{\Psi_\delta'(k-n)} =\delta\Big(\frac{e^{(n-k)\delta}-1}{e^{n\delta}-1}\Big)^2e^{k\delta}
   =\delta\Big(\frac{\sinh(\frac{n-k}{2}\delta)}{\sinh(\frac{n}{2}\delta)}\Big)^2.
 \end{align}
So, we see that (cf. definition (\ref{eq:psi1}))
\begin{equation*}
1+2\dot\psi_\delta(z)=\sum_{k=1}^{n}b^k(z).
\end{equation*}
Next, Lemma \ref{lem:(f^k)'} and the assumption $n|\delta|\leeq2$ lead to
\begin{equation}\label{eq:18}
\Big|\frac{b^k(z)}{\tilde b^k_n}-1\Big|<\ve, \;\textrm{ and }\;\; \Big|\sum_{k=1}^{n-\tilde n}b^k(z)-\sum_{k=1}^{n-\tilde n}\tilde b^k_n\Big|<\ve\sum_{k=1}^{n-\tilde n}\big|\tilde b^k_n\big|,
\end{equation}
where $n-k\greq\tilde n$, $0<|\delta|<\eta$, $\alpha=\arg\delta$, for suitably chosen $\tilde n\greq1$ and $\eta>0$.

{\em Step 2.}
We have $\sinh(z)=z+z^3/6+\ldots$. Thus, if $\arg z_1=\arg z_2$ and $|z_1|,|z_2|\leeq1$, then we can get
$$\arg(\sinh z_1)-\arg(\sinh z_2)<\frac\pi5.$$
Therefore we obtain
\begin{equation}\label{eq:|s|}
\Big|\sum_{k=1}^{n-\tilde n}\sinh^2\Big(\frac{n-k}{2}\delta\Big)\Big|>\cos\Big(\frac\pi5\Big)\sum_{k=1}^{n-\tilde n}\Big|\sinh^2\Big(\frac{n-k}{2}\delta\Big)\Big|,
\end{equation}
where $n|\delta|/2\leeq1$ and $\alpha=\arg\delta$.

{\em Step 3.}
Using (\ref{eq:|s|}), we conclude that
\begin{equation}\label{eq:sc2}
\bigg|\sum_{k=1}^{n-\tilde n}\sinh^2\Big(\frac{n-k}{2}\delta\Big)-\int_{0}^{n-\tilde n}\sinh^2\Big(\frac{n-x}{2}\delta\Big)dx\bigg| \leeq\ve\sum_{k=1}^{n-\tilde n}\Big|\sinh^2\Big(\frac{n-k}{2}\delta\Big)\Big|,
\end{equation}
where $0<|\delta|<\eta$ and $n>\tilde n$.

Since $2\sinh^2(z/2)=\cosh(z)-1$, we see that
\begin{equation}\label{eq:int2}
\int_{0}^{n-\tilde n}\sinh^2\Big(\frac{n-x}{2}\delta\Big)dx=\frac{1}{2\delta}\big(\sinh(n\delta)-n\delta-\sinh(\tilde n\delta)+\tilde n\delta\big).
\end{equation}
If we take $\tilde n=0$, then the above combined with (\ref{eq:|s|}) and (\ref{eq:sc2}) leads to the fact that
\begin{equation}\label{eq:23}
\sinh (z)-z\neq0 \;\;\;\;\;\textrm{ where }\;\;\;\;\; 0<|z|\leeq2.
\end{equation}
Thus, we can assume that $\sinh(\tilde n\delta)+\tilde n\delta$ is small with respect to $\sinh(n\delta)+ n\delta$, if $n>N$ where $N$ is large enough.

Next, using (\ref{eq:|s|}), (\ref{eq:sc2}) and (\ref{eq:int2}) (cf. proof of Proposition \ref{prop:psi}, {\em Step 2}), we can get
\begin{equation*}\label{eq:sh^2}
\bigg|\sum_{k=1}^{n-\tilde n}\sinh^2\Big(\frac{n-k}{2}\delta\Big)-\frac{1}{2\delta}\big(\sinh(n\delta)-n\delta\big)\bigg| \leeq4\ve\bigg|\sum_{k=1}^{n-\tilde n}\sinh^2\Big(\frac{n-k}{2}\delta\Big)\bigg|.
\end{equation*}
So, we conclude that (cf. definitions (\ref{eq:f}) and (\ref{eq:11}))
\begin{equation*}
\bigg|\sum_{k=1}^{n-\tilde n}\tilde b_n^k-\big(1+2\Gamma(n\delta)\big)\bigg| \leeq4\ve\bigg|\sum_{k=1}^{n-\tilde n}\tilde b_n^k\bigg|.
\end{equation*}
Thus, (\ref{eq:18}) and (\ref{eq:|s|}) lead to
\begin{equation}\label{eq:22}
\bigg|\sum_{k=1}^{n-\tilde n}b^k(z)-\big(1+2\Gamma(n\delta)\big)\bigg|\leeq6\ve\bigg|\sum_{k=1}^{n-\tilde n}\tilde b_n^k\bigg| \leeq\frac{6\ve}{1+4\ve}\big|1+2\Gamma(n\delta)\big|.
\end{equation}

{\em Step 4.}
Now, we have to estimate
\begin{equation*}
\bigg|\sum_{k=n-\tilde n+1}^{n}b^k(z) \bigg|= \bigg|\frac{\delta}{(f_\delta^{n-\tilde n})'(z)}\bigg(\sum_{k=1}^{\tilde n}\frac{1}{(f_\delta^{k})'(f_\delta^{n-\tilde n}(z))}\bigg)\bigg|.
\end{equation*}
Note that absolute value of the expression in bracket can be bounded above by a constant $K_1(\tilde n)$ (cf. Corollary \ref{cor:K}). Thus, Lemma \ref{lem:n^2} leads to
\begin{equation*}
\bigg|\sum_{k=n-\tilde n+1}^{n}b^k(z) \bigg|\leeq K_2(\tilde n)\frac{\delta}{(n-\tilde n)^2}.
\end{equation*}
We can assume that
\begin{equation*}
K_2(\tilde n)\frac{\delta}{(n-\tilde n)^2}\leeq \ve\frac{\sinh (n\delta)-n\delta}{\cosh (n\delta)-1}=\ve\big(1+2\Gamma(n\delta)\big),
\end{equation*}
where $0<|\delta|<\eta$ and $n>N$ for sufficiently chosen $\eta>0$ and $N\in\N$. Thus, the statement follows from (\ref{eq:22}) and the fact that $1+2\Gamma(z)\neq0$ where $0<z\leeq2$ (see (\ref{eq:23}) and definition (\ref{eq:f})).
\end{proof}

\section{Integral over ${B}_N$}\label{sec:bn}

The main result of this section is Proposition \ref{prop:BN}, which allows us to estimate the integral (\ref{eq:integral}), restricted to a set $ B_N$.

Note that the proof of  Proposition \ref{prop:BN} will be repeated after \cite[Lemma 7.3]{Jii}
(see also \cite[Proposition 4.1]{HZ}) with suitable changes.

First, we define a family of sets $\{A^{N_0}_{N,n}\}_{n\greq0}$, where $N\in\N$, $N_0\greq1$, which form a partition of
$ B_N\sms T^{-N_0}(\{1\}) $. 
Write
\begin{equation*}
   A^{N_0}_{N,n}:=
 \left\{
  \begin{array}{ll}
    T^{-N_0}(C_{N+n})\cap B_N, & \hbox{for $n\greq1$;} \\
    T^{-N_0}( B_N)\cap B_N, & \hbox{for $n=0$.}
 \end{array}
   \right.
\end{equation*}

\begin{lem}\label{lem:1}\cite[Lemma 5.1]{Jii}
For every $\alpha\in(-\pi/2,\pi/2)$ there exist $K>0$ and $\eta>0$ such that for every $N\in\N$, $N_0, n\greq1$ we have
$$\tilde{\mu_\delta}(A^{N_0}_{N,n})\leeq K N_0\,\tilde{\omega_\delta}( C_{N+n}),$$
where $0<|\delta|<\eta$ and $\alpha=\arg\delta$.
\end{lem}

The proof can be carried out exactly as in \cite{Jii} with one change. In our case it is not always true that $\omega_\delta(\mc C_{N+n}(\delta))>\omega_\delta(\mc C_{N+n+k}(\delta))$, where $k>0$. So we must change the argument in the last step of the proof. We know that
   $\omega_\delta(\mc C_{N+n}(\delta))/\omega_\delta(\mc C_{N+n+k}(\delta))\asymp|(f_\delta^k)'(z)|^{\mc D(\delta)}>\kappa$,
where $z\in\mc C_{N+n+k}(\delta)$ (see Corollary \ref{cor:K}). Thus there exist $K>0$, such that
   $$\sum_{k=1}^{N_0}\omega_\delta(\mc C_{N+n+k}(\delta))<K N_0\omega_\delta(\mc C_{N+n}(\delta)),$$
and we can apply this inequality.

\begin{prop}\label{prop:BN}
For every $\alpha\in(-\pi/2,\pi/2)$ there exists $\eta>0$ such that for every $N\in\N$ there exists $K(N)>0$ such that
$$\int_{B_N}\big|1+2\dot\varphi_\delta\big|d\tilde{\mu_\delta}< K(N)|\delta|,$$
where $0<|\delta|<\eta$ and $\alpha=\arg\delta$.
\end{prop}

\begin{proof}
Fix $\alpha\in(-\pi/2,\pi/2)$. Let $\eta>0$ be such that Lemmas \ref{lem:c/c2}, \ref{lem:n^2}, \ref{lem:1} and Corollary \ref{cor:mod} hold.

Fix $N\in\N$. Let $N_0\greq1$ and $s\in A^{N_0}_{N,n}$, then formula (\ref{eq:sumas1}) leads to
\begin{equation*}
1+2\dot\vp_\delta(s)=\sum_{k=1}^{N_0+n}
\frac{\delta}{(f_\delta^k)'(\vp_\delta(s))}+ \frac{2\dot\vp_\delta(T^{N_0+n}(s))+1}{(f_\delta^{N_0+n})'(\vp_\delta(s))}.
\end{equation*}
So, we divided $\dot\vp_\delta$ into two parts, the finite sum and the "tail".

Strategy of the proof is as follows.
First, in Step 1, we will prove that integral of the "tail" is less than $\frac12\int_{
B_N}|1+2\dot\vp_{\delta}|d\tilde{\mu_\delta}$ (for $N_0$ large enough, depending on $N$). Next, in Step 2, we will see
that integral of the finite sum is bounded by $K(N,N_0)|\delta|$, where $K(N,N_0)>0$ depends on $N$ and $N_0$. It means that
   $$\int_{B_N}\big|1+2\dot\vp_{\delta}\big|d\tilde{\mu_\delta}\leeq K(N,N_0)|\delta|+\frac12\int_{ B_N}\big|1+2\dot\vp_{\delta}\big|d\tilde{\mu_\delta}.$$
Since $N_0$ depends only on $N$, the assertion follows.

{\em Step 1.}
The measure $\tilde{\mu_\delta}$ is $T$-invariant, and $T^{N_0+n}(A^{N_0}_{N,n})\subset B_N$, hence
\begin{equation*}
\int_{A^{N_0}_{N,n}}\big|1+2\dot\vp_\delta(T^{N_0+n})\big|d\tilde{\mu_\delta}\leeq
\int_{B_N}\big|1+2\dot\vp_\delta\big|d\tilde{\mu_\delta}.
\end{equation*}
If $s\in A^{N_0}_{N,n}$, then $f_\delta^{N_0+n}(\vp_\delta(s))\in\mathcal B_N(\delta)$, so Lemma \ref{lem:n^2} and the above estimate give us
\begin{equation*}
\int_{A^{N_0}_{N,n}}\bigg|\frac{1+2\dot\vp_\delta(T^{N_0+n}(s))}{(f_\delta^{N_0+n})'(\vp_\delta(s))}\bigg|d\tilde{\mu_\delta}(s)
\leeq\frac{K_1(N)}{(N_0+n)^{2}}\int_{ B_N}\big|1+2\dot\vp_\delta(s)\big|d\tilde{\mu_\delta}(s).
\end{equation*}
Thus, we obtain
\begin{equation*}
\sum_{n=0}^\infty\int_{A^{N_0}_{N,n}}\bigg|\frac{1+2\dot\vp_\delta(T^{N_0+n})}{(f_\delta^{N_0+n})'(\vp_\delta)}\bigg|d\tilde{\mu_\delta}
\leeq\bigg(\sum_{n=0}^\infty\frac{K_1(N)}{(N_0+n)^{2}}\bigg)\int_{ B_N}\big|1+2\dot\vp_\delta\big|d\tilde{\mu_\delta},
\end{equation*}
and for $N_0$ large enough (depending on $N$), we have $\sum_{n=0}^\infty\frac{K_1(N)}{(N_0+n)^{2}}<\frac12$.

{\em Step 2.} Lemma \ref{lem:1}, Lemma \ref{lem:c/c2} and Corollary \ref{cor:mod} lead to
\begin{multline*}
\tilde{\mu_c}(A^{N_0}_{N,n})\leeq K_2\,N_0\,\tilde{\omega_c}(C_{N+n})\\
\leeq K_3\,N_0\,(\diam \mc C_{N+n}(\delta))^{\mc D(\delta)}\leeq K_4\,N_0\,(N+n)^{-2\mc D(\delta)},
\end{multline*}
where $n\greq1$. Using Corollary \ref{cor:K}, we obtain
\begin{multline*}
\int_{A^{N_0}_{N,0}}\Big|\sum_{k=1}^{N_0}\frac{\delta}{(f_\delta^k)'(\vp_\delta(s))}\Big|d\tilde{\mu_\delta}+
\sum_{n=1}^{\infty}\int_{A^{N_0}_{N,n}}\Big|\sum_{k=1}^{N_0+n} \frac{\delta}{(f_\delta^k)'(\vp_\delta(s))}\Big|d\tilde{\mu_\delta} \\
\leeq K_5N_0\,\tilde{\mu_\delta}(A^{N_0}_{N,0})|\delta| +\sum_{n=1}^{\infty}(N_0+n)\,K_6\,N_0\,(N+n)^{-2\mc D(\delta)}|\delta|\leeq  K(N,N_0)|\delta|,
\end{multline*}
where the constant $K(N,N_0)$ depends only on $N$ and $N_0$, as required.
\end{proof}

\section{Integral over ${M}_N$ and proof of Theorem \ref{thm:direction2}}\label{sec:mn}

Now we are going to prove Proposition \ref{prop:MN}, which is the crucial ingredient in the proof of Theorem \ref{thm:direction2}.

Let $\vartheta=\tan\alpha$, and let $v=e^{i\alpha}$, where $\alpha\in(-\pi/2,\pi/2)$. We define
\begin{equation}\label{eq:Delta}
\Delta(\alpha):=\frac{H_\mu}{\cos\alpha}\int_{0}^{\infty} \Big(\frac{x\sinh x+\vartheta x\sin\vartheta x}{\cosh x-\cos\vartheta x}-2\Big)
\Big(\frac{1/2}{\cosh x-\cos\vartheta x}\Big)^{\mc D(0)}dx,
\end{equation}
where $H_\mu$ is the constant from Lemma \ref{lem:est} and consequently from Lemma~\ref{lem:mm}. Note that $H_\mu/\cos\alpha>0$.

\begin{prop}\label{prop:MN}
For every $\alpha\in(-\pi/2,\pi/2)$ and $\ve>0$ there exist $N\in\N$ and $\eta>0$ such that
$$|\delta|^{2\mc D(\delta)-2}(\Delta(\alpha)-\ve)<\int_{ M_N}\re\Big(\frac{v+2v\dot\vp_{\delta}}{f'_{\delta}(\vp_{\delta})}\Big)d\tilde {\mu_\delta}<|\delta|^{2\mc D(\delta)-2}(\Delta(\alpha)+\ve),$$
where $0<|\delta|<\eta$ and $\alpha=\arg\delta$. Moreover $\Delta(\alpha)>0$ if $\alpha\in[-\pi/4,\pi/4]$.
\end{prop}

For $h>1$, $t>0$ and $\alpha\in(-\pi/2,\pi/2)$ let us write
\begin{equation}\label{eq:Q}
Q^h_\alpha(t):=H_\mu\re(v+2v\Gamma(vt))\int_{t}^{\infty}\Lambda_0^{h}(vs)ds,
\end{equation}
where $v=e^{i\alpha}$.
We know from (\ref{eq:2G+1}) that $1+2\Gamma(z)=z/3+O(z^3)$. The function $\Gamma$ is bounded on each ray $\mc R(\alpha)$,
so there exist a constant $K>0$ (depending on $\alpha$) for which
\begin{equation}\label{eq:r}
|\re(v+2v\Gamma(z))|\leeq|1+2\Gamma(z)|<K\min(|z|,1).
\end{equation}
Thus, using (\ref{eq:lll}), we see that there exists $\hat K>0$ such that
\begin{equation}\label{eq:Q1}
 \begin{array}{ll}
   |Q^h_\alpha(t)|<\hat K e^{-t\,h\cos\alpha} &\textrm{ for }\;\;t\in(1,\infty),\\
   |Q^h_\alpha(t)|<\hat K t^{-2h+2} &\textrm{ for }\;\;t\in(0,1].
 \end{array}
\end{equation}

Proposition \ref{prop:MN} is an immediate consequence of two following lemmas:

\begin{lem}\label{lem:r}
For every $\alpha\in(-\pi/2,\pi/2)$ and $\ve>0$ there exist $\eta>0$ and $N\in\N$ such that
\begin{equation*}
\int_0^\infty Q^{\mc D(\delta)}_\alpha(t)\,dt-\ve<|\delta|^{-2\mc D(\delta)+2}\int_{ M_N}\re\Big(\frac{v+2v\dot\vp_{\delta}}{f'_{\delta}(\vp_{\delta})}\Big)d\tilde {\mu_\delta}
<\int_0^\infty Q^{\mc D(\delta)}_\alpha(t)\,dt+\ve,
\end{equation*}
where $0<|\delta|<\eta$, $\alpha=\arg\delta$ and $v=e^{i\alpha}$.
\end{lem}

\begin{lem}\label{lem:limQ}
For every $\alpha\in(-\pi/2,\pi/2)$, we have
\begin{equation*}
\lim_{\delta\rightarrow0}\int_{0}^{\infty}Q^{\mc D(\delta)}_\alpha(t)\,dt=\int_{0}^{\infty}Q^{\mc D(0)}_\alpha(t)\,dt=\Delta(\alpha),
\end{equation*}
where $\alpha=\arg\delta$.
Moreover $\Delta(\alpha)>0$ for $\alpha\in[-\pi/4,\pi/4]$.
\end{lem}

Before proving Lemmas \ref{lem:r}, \ref{lem:limQ}, we state the following fact:

\begin{lem}\label{lem:vp-psi}
For every $\alpha\in(-\pi/2,\pi/2)$ there exist $K>0$ and $\eta>0$ such that for every $N\greq1$ we have
$$\int_{ M_N^*}|\dot\vp_\delta-\dot\psi_\delta(\vp_\delta)|d\tilde {\mu_\delta}<K\frac{|\delta|}{N},$$
where $0<|\delta|<\eta$ and $\alpha=\arg\delta$.
\end{lem}

\begin{proof}
Formula (\ref{eq:sumas1}) and definition (\ref{eq:psi1}) lead to
   $$\int_{ M_N^*}|\dot\vp_\delta-\dot\psi_\delta(\vp_\delta)|d\tilde {\mu_\delta}=\sum_{n=N+1}^{\infty}\int_{ C_n}\Big|\frac{\dot\vp_\delta(T^m)+1/2}{(f_\delta^n)'(\vp_\delta)}\Big|d\tilde {\mu_\delta}.$$
If $s\in C_n$ then $f_\delta^n(\vp_\delta(s))\in \mc B_0(\delta)$, therefore Lemma \ref{lem:n^2} gives us
   $$\int_{ M_N^*}|\dot\vp_\delta-\dot\psi_\delta(\vp_\delta)|d\tilde {\mu_\delta}<
   \sum_{n=N+1}^{\infty}\frac{K(0)}{n^2}\int_{ C_n}\Big|\dot\vp_\delta(T^n)+\frac12\Big|d\tilde {\mu_\delta}.$$
We have $T^n(C_n)\subset B_0$. So, because the measure $\tilde{\mu_\delta}$ is $T$-invariant, we obtain
   $$\int_{ M_N^*}|\dot\vp_\delta-\dot\psi_\delta(\vp_\delta)|d\tilde {\mu_\delta}<
   \sum_{n=N+1}^{\infty}\frac{K(0)}{n^2}\int_{ B_0}\Big|\dot\vp_\delta+\frac12\Big|d\tilde {\mu_\delta}.$$
Since $\sum_{n>N} 1/n^{2}<1/N$, the assertion follows from Proposition \ref{prop:BN}.
\end{proof}

Now we are in a position to prove Lemma \ref{lem:r}, Lemma \ref{lem:limQ}, and consequently Proposition \ref{prop:MN}. Next, we will prove Theorem \ref{thm:direction2}.

\begin{proof}[\textbf{Proof of lemma \ref{lem:r}}]
{\em Step 1.} Fix $\ve>0$ (small enough). We can find $N_0\in\N$ and $\eta>0$ such that Proposition \ref{prop:psi2} gives us
\begin{multline*}
\sum_{n=N+1}^{[1/|\delta|]}\int_{ C_n}\big|\big(1+2\dot\psi_\delta(\vp_\delta)\big)-\big(1+2\Gamma(n\delta)\big)\big|d\tilde {\mu_\delta}
<\ve\sum_{n=N+1}^{[1/|\delta|]}\big|1+2\Gamma(n\delta)\big|\tilde {\mu_\delta}( C_n),
\end{multline*}
where $N\greq N_0$ and $0<|\delta|<\eta$.

Next, possibly changing $\eta>0$, we conclude from Proposition \ref{prop:psi} that
\begin{multline*}
\sum_{n>[1/|\delta|]}\int_{ C_n}\big|2\dot\psi_\delta(\vp_\delta)-2\Gamma(n\delta)\big|d\tilde {\mu_\delta}
<\sum_{n>[1/|\delta|]}\big(e^{\ve n|\delta|}-1+\ve\big)\big|2\Gamma(n\delta)\big|\tilde {\mu_\delta}( C_n).
\end{multline*}
Since $\Gamma$ is bounded on each ray $\mc R_\alpha$, where $\alpha\in(-\pi/2,\pi/2)$, Corollary \ref{cor:Kdelta} and Lemma \ref{lem:K} lead to
\begin{equation}\label{eq:ekm}
\sum_{n>[1/|\delta|]}\big(e^{\ve n|\delta|}-1+\ve\big)\big|\Gamma(n\delta)\big|\tilde {\mu_\delta}( C_n)<\ve K_1\tilde {\mu_\delta}( M_{[1/|\delta|]})<\ve K_2 |\delta|^{2\mc D(\delta)-2},
\end{equation}
where $K_2>0$ does not depend on $\ve>0$ (but depends on $\alpha$).

Therefore, the above estimates give us
\begin{multline}\label{eq:111}
\sum_{n>N}\int_{ C_n}\big|\big(1+2\dot\psi_\delta(\vp_\delta)\big)-\big(1+2\Gamma(n\delta)\big)\big|d\tilde {\mu_\delta}\\
<\ve\sum_{n=N+1}^{[1/|\delta|]}\big|1+2\Gamma(n\delta)\big|\tilde {\mu_\delta}( C_n)+\ve K_2 |\delta|^{2\mc D(\delta)-2}.
\end{multline}

{\em Step 2.} Let $v=e^{i\alpha}$. Since $\mc D(\delta)<3/2$, possibly changing $\eta>0$, we conclude from Lemma \ref{lem:vp-psi} that
\begin{multline*}
\bigg|\int_{ M_N}\re\Big(\frac{v+2v\dot\vp_{\delta}}{f'_{\delta}(\vp_{\delta})}\Big)d\tilde {\mu_\delta}-\int_{ M_N^*}\re\Big(\frac{v+2v\dot\psi_{\delta}(\vp_{\delta})}{f'_{\delta}(\vp_{\delta})}\Big)d\tilde {\mu_\delta}\bigg| \\ <K_3\frac{|\delta|}{N}<\ve|\delta|^{2\mc D(\delta)-2},
\end{multline*}
where $0<|\delta|<\eta$. We can assume that $f'_{\delta}$ is close to $1$ on the set $\mc M_{N_0}(\delta)$ (cf. Corollary \ref{cor:M}), so using (\ref{eq:111}) we obtain
\begin{multline*}
\bigg|\int_{ M_N^*}\re\Big(\frac{v+2v\dot\psi_{\delta}(\vp_{\delta})}{f'_{\delta}(\vp_{\delta})}\Big)d\tilde {\mu_\delta}-\int_{ M_N^*}\re\big(v+2v\dot\psi_{\delta}(\vp_{\delta})\big)d\tilde {\mu_\delta}\bigg|\\ <\ve\int_{ M_N^*}\big|v+2v\dot\psi_{\delta}(\vp_{\delta})\big|d\tilde {\mu_\delta}<2\ve\sum_{n>N}\big|1+2\Gamma(n\delta)\big|\tilde {\mu_\delta}( C_n)+\ve |\delta|^{2\mc D(\delta)-2}.
\end{multline*}
Of course we have
\begin{multline*}
\bigg|\int_{ M_N^*}\re\big(v+2v\dot\psi_{\delta}(\vp_{\delta})\big)d\tilde {\mu_\delta}-\sum_{n>N}\re\big(v+2v\Gamma(n\delta)\big)\tilde {\mu_\delta}( C_n)\bigg|\\ <\sum_{n>N}\int_{ C_n}\big|\big(1+2\dot\psi_\delta(\vp_\delta)\big)-\big(1+2\Gamma(n\delta)\big)\big|d\tilde {\mu_\delta},
\end{multline*}
so the above estimates and (\ref{eq:111}) lead to
\begin{multline}\label{eq:33}
\bigg|\int_{ M_N}\re\Big(\frac{v+2v\dot\vp_{\delta}}{f'_{\delta}(\vp_{\delta})}\Big)d\tilde {\mu_\delta}-\sum_{n>N}\re\big(v+2v\Gamma(n\delta)\big)\tilde {\mu_\delta}( C_n)\bigg| \\<
3\ve\sum_{n>N}\big|1+2\Gamma(n\delta)\big|\tilde {\mu_\delta}( C_n)+\ve (K_2+2) |\delta|^{2\mc D(\delta)-2}.
\end{multline}

{\em Step 3.}
We conclude from (\ref{eq:r}) that
\begin{equation}\label{eq:r1}
|\re(v+2v\Gamma(n\delta))|\leeq|1+2\Gamma(n\delta)|<K_4\min(n|\delta|,1).
\end{equation}
So Lemma \ref{lem:mm} (\ref{lit:mm1}) and (\ref{lit:mm2}) leads to
\begin{multline}\label{eq:rr}
\bigg|\sum_{n=N+1}^{\infty}\re\big(v+2v\Gamma(n\delta)\big)\tilde {\mu_\delta}( C_n)\\
-|\delta|^{2\mc D(\delta)-1}H_\mu \sum_{n=N+1}^{\infty}\re\big(v+2v\Gamma(n\delta)\big) \int_{n|\delta|}^{\infty}\Lambda_0^{\mc D(\delta)}(vs)ds\bigg|<2\ve K_4|\delta|^{2\mc D(\delta)-2},
\end{multline}
for sufficiently chosen $N_0\in\N$, $\eta>0$, where $N\greq N_0$ and $0<|\delta|<\eta$.

Next (\ref{eq:r1}), Corollaries \ref{cor:meas} (\ref{cit:meas2}), \ref{cor:Kdelta}, and the fact that $\mc D(0)<3/2$ give us
\begin{multline*}
3\ve\sum_{n>N}\big|1+2\Gamma(n\delta)\big|\tilde {\mu_\delta}(C_n)<3\ve K_4\sum_{n>N}\min(n|\delta|,1)\tilde {\mu_\delta}( C_n)
\\<\ve K_5\sum_{n>N}^{[1/|\delta|]}n|\delta|\,n^{-2\mc D(\delta)+1}+\ve K_6\tilde {\mu_\delta}(\mc M_{[1/|\delta|]+1})
<\ve K_7|\delta|^{2\mc D(\delta)-2}.
\end{multline*}

So, we conclude from (\ref{eq:33}) combined with (\ref{eq:rr}), definition (\ref{eq:Q}), and the above estimate, that
\begin{multline}\label{eq:rrc}
\bigg|\int_{ M_N}\re\Big(\frac{v+2v\dot\vp_{\delta}}{f'_{\delta}(\vp_{\delta})}\Big)d\tilde {\mu_\delta}
-|\delta|^{2\mc D(\delta)-1}\sum_{n=N+1}^{\infty}Q^{\mc D(\delta)}_\alpha(n|\delta|)\bigg|\\
<\ve (K_7+K_2+2+2K_4)|\delta|^{2\mc D(\delta)-2}=\ve K_8|\delta|^{2\mc D(\delta)-2}.
\end{multline}

{\em Step 4.}
Denote by $V_x(Q^{\mc D(\delta)}_\alpha)$ the variation of the function $Q^{\mc D(\delta)}_\alpha$ on the set $[x,\infty)$.
It is easy to see that $V_x(Q^{\mc D(\delta)}_\alpha)$ is bounded for every $x>0$, whereas there exists $K_{9}>0$ such that $V_x(Q^{\mc D(\delta)}_\alpha)<K_{9}x^{2-2\mc D(\delta)}$, where $x\in(0,1)$ (cf. (\ref{eq:Q1})). Since we can assume that $N|\delta|<1$, the above estimate gives us
\begin{multline}\label{eq:rrr}
|\delta|^{2\mc D(\delta)-1}\bigg|\sum_{n=N+1}^{\infty}Q^{\mc D(\delta)}_\alpha(n|\delta|)-\int_{N}^{\infty}Q^{\mc D(\delta)}_\alpha(t|\delta|)\,dt\bigg|\\
<|\delta|^{2\mc D(\delta)-1}V_{N|\delta|}(Q^{\mc D(\delta)}_\alpha)<K_{9}N^{2-2\mc D(\delta)}|\delta|.
\end{multline}

Next, using (\ref{eq:Q1}) we get
\begin{equation*}
|\delta|^{2\mc D(\delta)-1} \bigg|\int^{N}_{0}Q^{\mc D(\delta)}_\alpha(t|\delta|)\,dt\bigg|<K_{10}N^{3-2\mc D(\delta)}|\delta|.
\end{equation*}
Thus (\ref{eq:rrc}) combined with (\ref{eq:rrr}) and the above inequality, leads to
\begin{equation}\label{eq:rrrr}
\bigg|\int_{ M_N}\re\Big(\frac{v+2v\dot\vp_{\delta}}{f'_{\delta}(\vp_{\delta})}\Big)d\tilde {\mu_\delta}\\
-|\delta|^{2\mc D(\delta)-1} \int_{0}^{\infty}Q^{\mc D(\delta)}_\alpha(t|\delta|)\;dt\bigg|<\ve K_{11}|\delta|^{2\mc D(\delta)-2},
\end{equation}
where $0<|\delta|<\eta$, for sufficiently chosen $\eta>0$ (depending on $N$).

Let $t|\delta|=u$. Then $dt=|\delta|^{-1}du$, so we get
\begin{equation*}
|\delta|^{2\mc D(\delta)-1}\int_{0}^{\infty}Q^{\mc D(\delta)}_\alpha(t|\delta|)\;dt=|\delta|^{2\mc D(\delta)-2}
\int_{0}^{\infty}Q^{\mc D(\delta)}_\alpha(u)\;du.
\end{equation*}
Thus the statement follows from (\ref{eq:rrrr}).
\end{proof}

\begin{proof}[\textbf{Proof of Lemma \ref{lem:limQ}}]
Fix $\alpha\in(-\pi/2,\pi/2)$. 

{\em Step 1.}
Since $\mc D(0)<3/2$, we can find $K>0$ and $\eta>0$ such that (\ref{eq:Q1}) holds for $h=\mc D(\delta)$ and $\hat K=K$, where $0<|\delta|<\eta$.
Thus, dominated convergence theorem leads to
\begin{equation}\label{eq:limQ}
\lim_{\delta\rightarrow0}\int_{0}^{\infty}Q^{\mc D(\delta)}_\alpha(t)\,dt=\int_{0}^{\infty}Q^{\mc D(0)}_\alpha(t)\,dt.
\end{equation}

{\em Step 2.}
We have to compute
\begin{equation*}
\int_{0}^{\infty}Q^{\mc D(0)}_\alpha(t)\;dt=
H_\mu \int_{0}^{\infty}\int_{t}^{\infty}\re\big(v+2v\Gamma(vt)\big) \Lambda_0^{\mc D(0)}(vs) ds\;dt.
\end{equation*}
Changing the order of integration, we obtain
\begin{equation*}
\int_{0}^{\infty}Q^{\mc D(0)}_\alpha(t)\;dt=H_\mu \int_{0}^{\infty}\int_{0}^{s}\Lambda_0^{\mc D(0)}(vs)\re\big(v+2v\Gamma(vt)\big) dt\;ds.
\end{equation*}
Next, using (\ref{eq:A}) and (\ref{eq:2G+1}), we see that
\begin{equation}\label{eq:7r}
\int_{0}^{\infty}Q^{\mc D(0)}_\alpha(t)\;dt=H_\mu \int_{0}^{\infty}\Big|\frac{1/4}{\sinh^2(vs/2)}\Big|^{\mc D(0)}\int_{0}^{s}\re\Big(v\frac{\sinh (vt)-vt}{\cosh (vt)-1}\Big) dt\;ds.
\end{equation}

{\em Step 3.}
Now we compute the inner integral from (\ref{eq:7r}).
First note that the substitution $z=vt$, where $t\in(0,s)$, gives us
\begin{equation*}
\int_{0}^{vs}\frac{\sinh z-z}{\cosh z-1} \,dz=\int_{0}^{s}\frac{\sinh (vt)-vt}{\cosh (vt)-1}\,v\, dt.
\end{equation*}
So, we conclude that
\begin{multline*}
\int_{0}^{s}\re\Big(v\frac{\sinh (vt)-vt}{\cosh (vt)-1}\Big) dt=\re\int_{0}^{s}\frac{\sinh (vt)-vt}{\cosh (vt)-1}\, v \, dt\\
=\re\int_{0}^{vs}\frac{\sinh z-z}{\cosh z-1}\, dz=\re\Big(\frac{z\sinh z}{\cosh z-1}\Big|_{0}^{vs}\Big)=
\re\Big(\frac{vs\sinh (vs)}{\cosh (vs)-1}\Big)-2.
\end{multline*}

{\em Step 4.}
Thus we have
\begin{equation*}
\int_{0}^{\infty}Q^{\mc D(0)}_\alpha(t)\;dt=H_\mu \int_{0}^{\infty} \Big(\re\Big(\frac{vs\sinh (vs)}{\cosh (vs)-1}\Big)-2\Big)
\Big|\frac{1/4}{\sinh^2(vs/2)}\Big|^{\mc D(0)}ds.
\end{equation*}
Let $x=\re(vs)=s\cos\alpha$, then $vs=x+i\vartheta x$ where $\vartheta=\tan\alpha$. So the above integral is equal to
\begin{equation*}
\frac{H_\mu}{\cos\alpha} \int_{0}^{\infty} \Big(\re\Big(\frac{(x+i\vartheta x)\sinh (x+i\vartheta x)}{\cosh (x+i\vartheta x)-1}\Big)-2\Big)
\Big|\frac{1/4}{\sinh^2(\frac{x+i\vartheta x}{2})}\Big|^{\mc D(0)}dx.
\end{equation*}
Next, using (\ref{eq:h3}) and (\ref{eq:h4}), we conclude that
\begin{multline}\label{eq:Q=}
\int_{0}^{\infty}Q^{\mc D(0)}_\alpha(t)\;dt\\
=\frac{H_\mu}{\cos\alpha} \int_{0}^{\infty} \Big(\frac{x\sinh x+\vartheta x\sin \vartheta x}{\cosh x-\cos\vartheta x}-2\Big)
\Big(\frac{1/2}{\cosh x-\cos\vartheta x}\Big)^{\mc D(0)}dx=\Delta(\alpha).
\end{multline}

{\em Step 5.}
Note that
\begin{multline*}
x\sinh x+\vartheta x\sin\vartheta x-2(\cosh x-\cos\vartheta x)
=\sum_{n=2}^{\infty}\Big(1-\frac1n\Big)\frac{1-(-1)^{n}\vartheta^{2n}}{(2n-1)!}x^{2n}.
\end{multline*}
Hence, if we assume that $|\vartheta|\leeq1$ (i.e. $\alpha\in[-\pi/4,\pi/4]$), then
\begin{equation*}\label{eq:6r}
\Delta(\alpha)>0.
\end{equation*}
So, the assertion follows from (\ref{eq:limQ}), (\ref{eq:Q=}) and the above inequality.
\end{proof}

\begin{proof}[\textbf{Proof of Theorem \ref{thm:direction2}}]
{\em Step 1.}
Fix $\alpha\in(-\pi/2,\pi/2)$ and $\ve>0$. Let $N\in\N$ and $\eta>0$ be such that Proposition \ref{prop:MN} holds.

Note that $|f_\delta(z)|$ is separated from $0$ for $z\in\mc J_\delta$, where $0<|\delta|<\eta$ and $\alpha=\arg\delta$. Moreover $\mc D(\delta)<3/2$, so we have $2\mc D(\delta)-2<1$. Thus, possibly changing $\eta>0$, we conclude from Proposition \ref{prop:BN} that
\begin{equation*}
\int_{ B_N}\Big|\frac{v+2v\dot\vp_{\delta}}{f'_{\delta}(\vp_{\delta})}\Big|d\tilde {\mu_\delta}<K(N)|\delta|<\ve|\delta|^{2\mc D(\delta)-2},
\end{equation*}
where $0<|\delta|<\eta$, $\alpha=\arg\delta$ and $v=e^{i\alpha}$.
Hence, Proposition \ref{prop:MN} leads to
\begin{equation*}
|\delta|^{2\mc D(\delta)-2}(\Delta(\alpha)-2\ve)<\int_{\partial \D}\re\Big(\frac{v+2v\dot\vp_{\delta}}{f'_{\delta}(\vp_{\delta})}\Big)d\tilde {\mu_\delta}<|\delta|^{2\mc D(\delta)-2}(\Delta(\alpha)+2\ve).
\end{equation*}

{\em Step 2.}
Since $\ve>0$ is arbitrary, and if $\delta=tv$ then $|\delta|=t$, we obtain
\begin{equation*}
\lim_{t\rightarrow0^+}t^{-2\mc D(\delta)+2} \int_{\partial \D}\re\Big(\frac{v+2v\dot\vp_{\delta}}{f'_{\delta}(\vp_{\delta})}\Big)d\tilde {\mu_\delta}=\Delta(\alpha).
\end{equation*}

Because $f_{tv}'(\vp_{tv})=(1+tv)+2\vp_{tv}$, Propositions \ref{prop:mu}, \ref{prop:CHm} give us
   $$\lim_{t\rightarrow0^+}\int_{\partial\D}\log|f_{tv}'(\vp_{tv})|d\tilde\mu_{tv}= \int_{\partial\D}\log|f_{0}'(\vp_{0})|d\tilde{\mu_{0}}=:\chi.$$
For $s\neq\pm1$ we have $|f_{0}'(\vp_{0}(s))|>1$, thus $\chi>0$.

Therefore, the fact $\mc D(tv)\rightarrow \mc D(0)$ and formula (\ref{eq:d}) lead to
\begin{equation}\label{eq:limd}
\lim_{t\rightarrow0^+}\frac{\mc D'_v(tv)}{t^{2\mc D(\delta)-2}}=\frac{-\mc D(0)}{\chi}\Delta(\alpha).
\end{equation}

{\em Step 3.}
Now we have to replace $2\mc D(\delta)-2$ by $2\mc D(0)-2$ in the exponent.
We can assume that $2\mc D(\delta)-2>0$, so (\ref{eq:limd}) leads to
\begin{equation*}
|\mc D(tv)-\mc D(0)|\leeq\int_{0}^{t}|\mc D'_v(sv)|ds\leeq K\int_{0}^{t}s^{2\mc D(\delta)-2}ds<Kt.
\end{equation*}
Thus we get
\begin{equation*}
1\greq t^{|\mc D(tv)-\mc D(0)|}> t^{K t}=e^{K t\log t}.
\end{equation*}
Since $t\log t\rightarrow0$ when $t\rightarrow0^+$, we obtain $t^{|\mc D(tv)-\mc D(0)|}\rightarrow1$. Therefore
\begin{equation}\label{eq:1}
\lim_{t\rightarrow0^+}\frac{t^{2\mc D(\delta)-2}}{t^{2\mc D(0)-2}}=1.
\end{equation}

Since $\vartheta=\tan\alpha$, we have $\sqrt{\vartheta^2+1}=1/\cos\alpha$. Thus, (\ref{eq:limd}), (\ref{eq:1}) and next definitions  (\ref{eq:Delta}), (\ref{eq:I}) lead to
   \begin{equation*}
      \lim_{t\rightarrow0^+}\frac{\mc D'_v(tv)}{t^{2\mc D(0)-2}}=\frac{-\mc D(0)}{\chi}\Delta(\alpha) =\frac{-\mc D(0)}{\chi}\big(-2^{-\mc D(0)}H_\mu\Omega(\vartheta)\big).
   \end{equation*}
Therefore, the required limit exists, and we see that
   $$\mc A=2^{-\mc D(0)}H_\mu\frac{\mc D(0)}{\chi}>0.$$
Because $\Delta(\alpha)$ and $\Omega(\vartheta)$, where $\vartheta=\tan\alpha$, has opposite signs, we see from Proposition \ref{prop:MN} that $\Omega(\vartheta)<0$ if $\alpha\in[-\pi/4,\pi/4]$. Thus the proof is finished.
\end{proof}

\section*{Acknowledgement}

The author would like to thank F. Przytycki and M. Zinsmeister for helpful discussions. The author was partially supported by National Science Centre grant 2014/13/B/ST1/01033 (Poland).

\appendix
\section{}\label{app:A}

Now we give proofs of Lemmas from Section \ref{sec:t}.
 First, let us recall that
\begin{equation*}
 \begin{array}{lll}
   \sinh z&=2\sinh(z/2)\cosh(z/2),\\
   \cosh z+1&=2\cosh^2(z/2),\\
   \cosh z-1&=2\sinh^2(z/2).
 \end{array}
\end{equation*}
If $z=x+iy$, then we have
\begin{equation*}
 \begin{array}{ll}
   \sinh z&=\sinh x\cos y+i\cosh x\sin y,\\
   \cosh z&=\cosh x\cos y+i\sinh x\sin y.
 \end{array}
\end{equation*}
So, we obtain
\begin{equation}\label{eq:h3}
\Big|\sinh\frac{z}{2}\Big|^2=\sinh^2\frac{x}{2}+\sin^2\frac{y}{2}=\cosh^2\frac{x}{2}-\cos^2\frac{y}{2}=\frac12(\cosh x-\cos y).
\end{equation}
Next, we conclude that
\begin{equation}\label{eq:h4}
\frac{\sinh z}{\cosh z -1}=\frac{\cosh(z/2)}{\sinh(z/2)}=\frac{\sinh x -i\sin y}{\cosh x -\cos y}.
\end{equation}

\begin{proof}[\textbf{Proof of Lemma \ref{lem:<}}]
{\em Step 1.} If $x:=e^{|z|}$, then the first statement follows from the fact that for $0<\ve<1$ the real roots of
$$x^2-(1+\ve)x+2\ve>0,$$
are less than 1 (if they exist).

{\em Step 2.} Note that
$$|XY-1|\leeq|X-1||Y-1|+|X-1|+|Y-1|.$$
We have
$$(e^{\ve_1 |z|}-1+\ve)+(e^{\tilde\ve_1 |z|}-1+\tilde\ve)<e^{(\ve_1+\tilde\ve_1) |z|}-1+(\ve+\tilde\ve).$$
Next, using the fact that $\ve, \tilde\ve\in(0,1)$, we can get
\begin{multline*}
(e^{\ve_1 |z|}-1+\ve)(e^{\tilde\ve_1 |z|}-1+\tilde\ve)\\<e^{(\ve_1+\tilde\ve_1) |z|}-1+(\ve+\tilde\ve)-e^{\ve_1|z|}(1-\tilde\ve)-e^{\tilde\ve_1|z|}(1-\ve)+2(1-\ve)(1-\tilde\ve)\\
<e^{(\ve_1+\tilde\ve_1) |z|}-1+(\ve+\tilde\ve).
\end{multline*}
Therefore
$$|XY-1|<2(e^{(\ve_1+\tilde\ve_1) |z|}-1+(\ve+\tilde\ve))<e^{(2\ve_1+2\tilde\ve_1) |z|}-1+(2\ve+2\tilde\ve),$$
and the second statement follows.

{\em Step 3.} Finally, assumption $|X-1|<\ve$ and the first statement lead to
\begin{multline*}
|Xe^{z}-1|\leeq  |Xe^{z}-X|+|X-1|
<|X||e^{z}-1|+\ve\\<(1+\ve)(e^{|z|}-1)+\ve
=e^{|z|}(1+\ve)-1< e^{2|z|}-1+2\ve,
\end{multline*}
and the proof is finished.
\end{proof}

\begin{proof}[\textbf{Proof of Lemma \ref{lem:1/2}}]
Let $\alpha=\arg\delta$. We can assume that $\alpha\in(0,\pi/2)$. Next, we can find $\eta>0$ such that $\tilde
m\im\delta<\pi$ for $0<|\delta|<\eta$. We will consider two cases:

$\textbf{-}$ If $m\im\delta<\pi$, then for every $\tilde m\leeq k\leeq m$ we have $k|\delta|<\pi$. Therefore
$-\pi<\im(-k\delta)<0$, and then $-\pi<\arg(e^{-k\delta}-1)<-\pi/2$. Thus the assertion follows.

$\textbf{-}$ If $m\im\delta\greq\pi$, then using the fact that $\re(e^{-k\delta}-1)<0$, we can get
   $$-\re\sum_{k=\tilde m}^{m}\big(e^{-k\delta}-1\big)> \sum_{k=\tilde m}^{m}\big|\im\big(e^{-k\delta}-1\big)\big|,$$
and the assertion follows.
\end{proof}

\begin{proof}[\textbf{Proof of Lemma \ref{lem:w}}]
We can assume that $\alpha\in(0,\pi/2)$ where $\alpha=\arg w$. Let $\vartheta:=\tan\alpha$, then $w=t+i\vartheta t$, for $t\in\R^+$. Using (\ref{eq:h4}), we obtain
$$\frac{1}{e^{w}-1}+\frac12=\frac12\frac{\cosh(w/2)}{\sinh(w/2)}=\frac12\Big(\frac{\sinh t}{\cosh t-\cos \vartheta t}-i\frac{\sin \vartheta t}{\cosh t-\cos \vartheta t}\Big),$$
and next
$$\Big|-\frac{\sin \vartheta t}{\cosh t-\cos \vartheta t}\cdot\frac{\cosh t-\cos \vartheta t}{\sinh t}\Big| =\Big|\frac{\sin \vartheta t}{\sinh t}\Big|< \vartheta=\tan\alpha.$$
So, we conclude that $\arg ((e^{w}-1)^{-1}+1/2)\in(-\alpha,\alpha)$.

In order finish the proof we have to estimate $(1/2)\sinh t(\cosh t-\cos \vartheta t)^{-1}$ from below.

First we assume that $\vartheta\in(0,1]$ (i.e. $\alpha\in(0,\pi/4]$). So, using (\ref{eq:h4}), we get
$$\frac{1/2\sinh t}{\cosh t-\cos \vartheta t}\greq\frac{1/2\sinh t}{\cosh t-(2-\cosh \vartheta t)}\greq\frac14\frac{\sinh t}{\cosh t-1}=\frac{\cosh (t/2)}{4\sinh (t/2)}>\frac14.$$
Thus the statement holds in the case $\vartheta=\tan\alpha\leeq1$.

Let us assume that $\vartheta>1$ (i.e. $\alpha\in(\pi/4,\pi/2)$). If $t\greq4/(3\vartheta)$, then
$$\frac{1/2\sinh t}{\cosh t-\cos \vartheta t}\greq \frac{1/2\sinh t}{\cosh t+1}=\frac12\tanh\frac{t}{2}\greq\frac12\tanh\frac{2}{3\vartheta}>\frac{1}{4\vartheta}=\frac{1}{4}\cot\alpha.$$
If $0<t<4/(3\vartheta)$, then
$$\frac{1/2\sinh t}{\cosh t-\cos \vartheta t}>\frac{t/2}{(1+t^2)-(1-\vartheta^2t^2/2)}=\frac{1/2}{t(1+\vartheta^2/2)}>\frac34\frac{\vartheta}{2+\vartheta^2}>\frac{1}{4\vartheta},$$
and the statement follows.
\end{proof}

\begin{proof}[\textbf{Proof of Lemma \ref{lem:e}}]
{\em Step 1.}
Using Cauchy's mean value Theorem we can prove that
$$1-\ve<\frac{e^{(1-\ve)x}-1}{e^{x}-1}\;\textrm{ and }\;\frac{e^{(1+\ve)x}-1}{e^{x}-1}<1+\ve,$$
where $\ve\in(0,1)$, $x\in\R^-$. Since the above expressions are bounded by $1$ from above and from below
respectively, statement holds for $\tilde w\in\R^-$ and $\delta\in\R^+$.

{\em Step 2.}
Let us consider $\tilde w\in\C$ and $\delta\in\R^+$. It is easy to see that
$$\sup_{\tilde w\in B(w,\ve|w|)}|e^{\tilde w\delta}-e^{w\delta}|= e^{(w+\ve|w|)\delta}-e^{w\delta}.$$
Therefore
$$\exp(B(w\delta,\ve|w|\delta))\subset B(e^{w\delta},e^{(w+\ve|w|)\delta}-e^{w\delta}).$$
So, if $\tilde w\in B(w,\ve|w|)$, we get
$$e^{\tilde w\delta}-1\in B(e^{w\delta}-1,e^{(w+\ve|w|)\delta}-e^{w\delta}),$$
and then
$$\frac{e^{\tilde w\delta}-1}{e^{w\delta}-1}-1\in B\Big(0,\frac{ e^{(w+\ve|w|)\delta}-e^{w\delta}}{|e^{w\delta}-1|}\Big).$$

Since $w+\ve|w|\in\R^-$ and $w+\ve|w|\in B(w,\ve|w|)$ we already know that
$$1-\ve<\frac{e^{(w+\ve|w|)\delta}-1}{e^{w\delta}-1}<1+\ve.$$
So
$$\Big|\frac{ e^{(w+\ve|w|)\delta}-e^{w\delta}}{e^{w\delta}-1}\Big|=\Big|\frac{e^{(w+\ve|w|)\delta}-1}{e^{w\delta}-1}-1\Big|<\ve.$$

{\em Step 3.}
Let $\arg\delta=\alpha\in(-\pi/2,\pi/2)$, and let $\tilde w\in B(w,\ve|w|\cos\alpha)$. Similarly as before we get
\begin{multline*}
\sup_{\tilde w\in B(w,\ve|w|\cos\alpha)}|e^{\tilde w\delta}-e^{w\delta}|=|e^{w\delta+\ve|w||\delta|\cos\alpha}-e^{w\delta}|\\ =|e^{w\delta}|(e^{\ve|w|\re\delta}-1)=(e^{(w+\ve|w|)\re\delta}-e^{w\re\delta}),
\end{multline*}
and then
$$\frac{e^{\tilde w\delta}-1}{e^{w\delta}-1}-1\in B\Big(0,\frac{ e^{(w+\ve|w|)\re\delta}-e^{w\re\delta}}{|e^{w\delta}-1|}\Big).$$
Thus, desired inequality follows from the previous case, and the fact that $|e^{w\delta}-1|>|e^{w\re\delta}-1|$.

{\em Step 4.} Finally, the inequality for $w$ replaced by $\tilde w$ and vice versa follows from the fact that if $z\in B(1,\ve/2)$ and $\ve\in(0,1)$, then $1/z\in B(1,\ve)$.
\end{proof}

\begin{proof}[\textbf{Proof of Lemma \ref{lem:Psi'/}}]
Fix $\ve>0$. Using Lemma \ref{lem:e} we can find a constant $K(\alpha)>2$ such that
  \begin{equation*}
     \Big|\frac{e^{w\delta}-1}{e^{\tilde w\delta}-1}-1\Big|<\frac{\ve}{8},
  \end{equation*}
where $|w-\tilde w|<\ve|w|/K(\alpha)$ (i.e. $\tilde w\in B(w,\ve|w|/K(\alpha))$). We can also assume that the above inequality holds after interchanging the roles of $w$ and $\tilde w$.
Next, Lemma \ref{lem:<} (\ref{lit:<2}) leads to
  \begin{equation}\label{eq:k4}
    \Big|\Big(\frac{e^{w\delta}-1}{e^{\tilde w\delta}-1}\Big)^2-1\Big|<\frac\ve2.
  \end{equation}
Thus, Lemma \ref{lem:<} (\ref{lit:<3}) gives us
  \begin{equation*}
     \Big|\Big(\frac{e^{w\delta}-1}{e^{\tilde w\delta}-1}\Big)^2e^{(\tilde w-w)\delta}-1\Big|<e^{2|(\tilde w-w)\delta|}-1+\ve.
  \end{equation*}
Since $|(w-\tilde w)\delta|<\ve|w\delta|/K(\alpha)<\ve|w\delta|/2$, the assertion follows from (\ref{eq:Psi'2}).
\end{proof}

\end{document}